\documentclass[a4paper,11pt,notitlepage]{article}

\title{The stationary distribution of a Markov jump process glued together from two state spaces at two vertices\footnote{This is a revision of our Author's Original Manuscript submitted for publication to Stochastic Models, Taylor \& Francis LLC.}}

\author{Bence M\'elyk\'uti\footnote{Department of Mathematical Stochastics, University of Freiburg, Eckerstr.~1, 79104~Freiburg, Germany} \footnote{Center for Biological Systems Analysis (ZBSA), University of Freiburg, Habsburgerstr.~49, 79104~Freiburg, Germany} \footnote{Corresponding author. Email: \texttt{melykuti@stochastik.uni-freiburg.de}.}, Peter Pfaffelhuber\textsuperscript{$\dag$}}

\usepackage{amssymb} 
\usepackage{amsmath}
\usepackage[amsmath,thmmarks]{ntheorem}
\usepackage[cp1250]{inputenc} 
\usepackage[T1]{fontenc}
\usepackage[UKenglish]{babel}
\reversemarginpar
\usepackage{graphicx}

\usepackage[sort&compress,square,super]{natbib}

\theorembodyfont{\normalfont}

\newtheorem{thm}{Theorem}
\newtheorem{prop}[thm]{Proposition}

\theoremstyle{nonumberplain}
\theoremsymbol{\ensuremath{\square}}
\newtheorem{proof}{Proof}

\newcommand{\id}{\,\mathrm{d}} 

\newcommand{\E}{\mathrm{E}}
\renewcommand{\P}{\mathrm{P}}

\newcommand{\T}{^{\mathrm{T}}}

\begin{document}

\maketitle

\begin{abstract}
We compute the stationary distribution of a continuous-time Markov chain which is constructed by gluing together two finite, irreducible Markov chains by identifying a pair of states of one chain with a pair of states of the other and keeping all transition rates from either chain (the rates between the two shared states are summed). The result expresses the stationary distribution of the glued chain in terms of quantities of the two original chains. Some of the required terms are nonstandard but can be computed by solving systems of linear equations using the transition rate matrices of the two original chains. Special emphasis is given to the cases when the stationary distribution of the glued chain is a multiple of the equilibria of the original chains, and when not, for which bounds are derived.
\end{abstract}

\section{Introduction}\label{s:intro}
Computing the stationary distribution of an irreducible continuous-time Markov chain on a finite state space is easy in principle. If $Q$ is the transition rate matrix, it only requires finding a probability vector~$\pi$ that solves $\pi\T Q=0$. (This is a standard result in the classical literature on Markov chains, see e.g.\ \citet{Liggett2010, Norris1998}.) Elementary examples include birth--death chains~\cite{KarlinTaylor1975} or a circular state space~\cite{Adan_Resing_1996}. However, if the transition graph of the Markov chain is more complicated, hardly any general result is known about the shape of the equilibrium. (A notable exception is the Markov chain tree theorem \cite{Leighton_Rivest_1986}.) We are interested whether it is possible to compute the stationary distribution from the stationary distributions of smaller parts of the state space recursively. The motivation for this problem comes from~\citet*{Melykuti_Hespanha_Khammash_2014}, where Markov chains arising from biochemical reaction networks were studied. In this paper, we approach the challenge of computing the stationary distribution of a Markov chain which is obtained by gluing together two Markov chains with simpler transition graphs at two states. (Gluing at a single state is also discussed.)

The standard model of biochemical kinetics represents the reacting system with a dynamical system where each coordinate of the state vector is the number of molecules present of a reacting chemical species. There is great interest in understanding how biochemical modules behave when they are connected (see Del Vecchio's work on retroactivity~\cite{Gyorgy_DelVecchio_2014}). Connecting two biochemical modules is represented by merging $\mathbb{N}^k$ with $\mathbb{N}^\ell$ into an $\mathbb{N}^m$, where $m<k+\ell$ (the shared chemical species are written only once). While we are far from addressing this question, the thinking of this paper points broadly in a similar direction.

A question more classical than connecting biochemical modules is how currents or the total resistance of electrical circuits change when two circuits are connected. Connecting state spaces and studying probability flows, how it is done in this paper, have a similar flavour.

The basic idea of our approach is to use a regenerative structure of the glued Markov chain. If $1$ and~$2$ denote the states which were glued together from the single chains, consider excursions from state~$1$ to~$2$ and back in the glued chain. Such excursions always happen within the single chains, which have a simpler structure. By combining the probabilities of excursions with their lengths, we are able to give the equilibrium of the glued chain in terms of the original chains (Theorem~\ref{th:main}). As a main tool, we use the law of large numbers for regenerative processes (see e.g. \citet{Smith1955,Serfozo2009} or \citet{Roginsky1994} and references therein). Then, in Section~\ref{s:recursion}, we discuss how to apply our main result in practice. A special case arises when the equilibrium of the combined chain is a multiple of the equilibria of the single chains, examined in detail in Section~\ref{s:parallel}. Section~\ref{s:examples} demonstrates some of the results on two related examples.\\

Let us assume that two irreducible, time-homogeneous, continuous-time Markov chains with finite state spaces, $\mathcal{X}^A=(\mathcal{X}^A_t)_{t\ge 0}$ and~$\mathcal{X}^B=(\mathcal{X}^B_t)_{t\ge 0}$, are given. Let $\mathcal{X}^A$ have $r\ge 2$ states,
\[\mathcal{V}^A=\{-r+3,-r+4,\dots,0,1,2\},\]
while $\mathcal{X}^B$ has $s\ge 2$ states,
\[\mathcal{V}^B=\{1,2,\dots,s\}.\]
The transition rate matrices are $Q^A$ and~$Q^B$, respectively. For $i,j\in\mathcal{V}^A$, $i\neq j$, $Q^A_{ij}$ is the transition rate of $\mathcal{X}^A$ from state~$i$ to state~$j$, and $Q^A_{ii}=-\sum_{j\neq i} Q^A_{ij}$. For $i,j\in\mathcal{V}^B$, $Q^B_{ij}$ is defined analogously. 

We create a new Markov chain $\mathcal{X}=(\mathcal{X}_t)_{t\ge 0}$ by gluing together the two state spaces at two states: we identify state $1\in\mathcal{V}^A$ with $1\in\mathcal{V}^B$ (and call it state~$1$ in the glued chain) and $2\in\mathcal{V}^A$ with $2\in\mathcal{V}^B$ (to be denoted by~$2$), and keep all transitions that were present in $\mathcal{X}^A$ or in~$\mathcal{X}^B$. Between states $1$ and~$2$, transitions of both $\mathcal{X}^A$ and~$\mathcal{X}^B$ are retained: the transition rates add up. If originally both chains had a transition, say, from $1$ to~$2$, then we can think of the dynamics as there being a choice between two parallel edges. The glued chain $\mathcal{X}$ has transition rate matrix $Q\in\mathbb{R}^{(r+s-2)\times (r+s-2)}$: for $i,j\in\{-r+3,-r+4,\dots,s\}$, 
\[Q_{ij}=Q^A_{ij}\,\mathbf{1}(i,j\le 2)+Q^B_{ij}\,\mathbf{1}(i,j\ge 1).\]
The set of states of $\mathcal{X}$ that belonged to $\mathcal{X}^A$ with the exception of $\{1, 2\}$ is denoted by
\[\mathcal{S}^A=\mathcal{V}^A\setminus \{1,2\}=\{-r+3,-r+4,\dots,0\}.\]
\[\mathcal{S}^B=\mathcal{V}^B\setminus \{1,2\}=\{3,4,\dots,s\}\]
is defined similarly for~$\mathcal{X}^B$. Hence, the state space of $\mathcal{X}$ can be written as the disjoint union $\mathcal{S}^A\dot{\cup}\{1, 2\}\dot{\cup}\mathcal{S}^B$.

Our goal is to compute the stationary distribution $\pi$ of~$\mathcal{X}$, i.e.\ $\pi\T Q=0$, from the stationary distributions $\pi^A$ and~$\pi^B$ of $\mathcal{X}^A$ and, respectively,~$\mathcal{X}^B$. (Note that irreducibility implies that $\pi^A$, $\pi^B$ and~$\pi$ all exist uniquely and are all strictly positive.) However, it will turn out (see Theorem~\ref{th:main} below) that more information than $\pi^A$ and $\pi^B$ is needed to compute $\pi$. In theory, finding a left nullvector~$\pi$ to matrix~$Q$ is a basic task. Still, we hope to learn more about $\pi$ and~$\mathcal{X}$ by expressing~$\pi$ through $\mathcal{X}^A$ and~$\mathcal{X}^B$. This approach would also give a method to recursively compute the stationary distribution of a Markov chain on a large state space from properties of two smaller parts, parts thereof and so on. 

\subsubsection*{Gluing at one state only}
In~\citet{Melykuti_Hespanha_Khammash_2014}, the stationary distribution was expressed when the gluing of two state spaces happened at one state only. Here, with the obvious adjustment of notation, the rate matrices $Q^A\in\mathbb R^{r\times r}$ and $Q^B\in\mathbb R^{s\times s}$ translate into the rate matrix of the glued chain $Q \in \mathbb R^{(r+s-1)\times (r+s-1)}$ by
\[Q_{ij}=Q^A_{ij}\,\mathbf{1}(i,j\le 1)+Q^B_{ij}\,\mathbf{1}(i,j\ge 1)\]
for $i,j\in\{-r+2,-r+3,\dots,s\}$. The resulting stationary distribution is a constant multiple of~$\pi^A$ on~$\mathcal{S}^A$, and of~$\pi^B$ on~$\mathcal{S}^B$ as follows:
\begin{align}\label{eq:982}
\pi_i=\begin{cases}C\pi^A_i \pi^B_1,&\ \textrm{if}\ i\in\{-r+2,-r+3,\dots,0\},\\C\pi^A_1 \pi^B_1,&\ \textrm{if}\ i=1,\\C\pi^A_1 \pi^B_{i},&\ \textrm{if}\ i\in\{2,3,\dots,s\},\end{cases}
\end{align}
with $C=(\pi^A_1+\pi^B_1-\pi^A_1\pi^B_1)^{-1}$. This claim follows easily from the structure of the transition rate matrix~$Q$, or even from the Markov chain tree theorem, but there is an appealing alternative explanation. When the process~$\mathcal{X}$ leaves $\mathcal{S}^A$ for~$\mathcal{S}^B$ during its random walk, it can only do so via the shared state. The process cannot return at any other location but at the shared state. When it comes back, from the perspective of~$\mathcal{S}^A$, it is as if nothing has happened. What happens inside~$\mathcal{S}^B$ has no effect on the relative weighting of the states in the stationary distribution on~$\mathcal{S}^A$. If one disregards the time intervals spent in~$\mathcal{S}^B$, the behaviour of~$\mathcal{X}$ on~$\mathcal{S}^A$ is identical to that of~$\mathcal{X}^A$. From this vantage point, visits to the shared state can be seen as renewal times.

Ref.~\cite{Melykuti_Hespanha_Khammash_2014} noted also that this result allows the recursive computation of the stationary distribution in the case of transition graphs that arise by gluing together linear and circular graphs one by one, but always at one state at a time.

Establishing the stationary distribution when gluing at two states is more difficult: when the process leaves $\mathcal{S}^A$ at one shared state, it might come back via the other. As we shall see, one requires additional information about the two original Markov chains.

\subsubsection*{Examples for gluing at two states}
The forthcoming result allows the computation of the stationary distribution on any irreducible transition graph via the gluing of linear graphs onto a growing graph. A special case of interest is the gluing of a linear path of two states onto a graph because it is equivalent to adding new transitions between two states in a Markov chain or increasing their transition rates if nonzero rates were already defined. Another case of relevance is the gluing of the two ends of a three-state linear path onto a graph because it is equivalent to adding a new state to a Markov chain and connecting it to two pre-existing states.

Both these cases introduce a local perturbation to the Markov process. One might expect that the stationary distribution will change considerably only in a neighbourhood. On the other hand, equilibrium is the long-term behaviour of the process, and the effects of the local perturbation have an infinite amount of time to propagate. It would be interesting to know how these two aspects balance.

\section{The stationary distribution of the glued Markov chain}\label{s:mainresult}

We start by introducing the necessary notations and notions. Firstly,
\begin{align*}
q_{1A}&:=\sum_{j\neq 1} Q^A_{1j}=-Q^A_{11},&q_{1B}&:=\sum_{j\neq 1} Q^B_{1j}=-Q^B_{11},\\
q_{2A}&:=\sum_{j\neq 2} Q^A_{2j}=-Q^A_{22},&q_{2B}&:=\sum_{j\neq 2} Q^B_{2j}=-Q^B_{22}
\end{align*}
are the total rates of leaving the glued state $1$ or~$2$ via edges of chain $\mathcal{X}^A$ or~$\mathcal{X}^B$.

We define an \emph{excursion} to be a transition path of $\mathcal{X}$ (or of an original chain $\mathcal{X}^A$ or $\mathcal{X}^B$) from either state~$1$ or~$2$ that leaves the initial state, until the first time when $1$ or $2$ is entered. A direct transition from $1$ or~$2$ to the other is also an excursion.

The \emph{type} of an excursion consists of the information, written in superscript, which chain it is in ($\mathcal{X}$, $\mathcal{X}^A$ or~$\mathcal{X}^B$), and, written in subscript, where it starts from (state $1$ or~$2$), whether it uses transitions from $Q^A$ or $Q^B$, and where it ends ($1$ or~$2$). For instance, we talk about $t^A_{1A2}$- or $t_{2B1}$-excursions.

Accordingly, for $\mathcal{X}^A$, the probabilities of different types of excursions can be denoted by
\begin{align*}
p^A_{1A1}&:=\P(\sigma_{1}^A < \sigma_2^A\, |\, \mathcal{X}^A_0 = 1),\\
p^A_{1A2}&:=\P(\sigma_{2}^A < \sigma_{1}^A\, |\, \mathcal{X}^A_0 = 1),
\end{align*}
where $\sigma^{\delta}_{\varepsilon}$ is the first hitting time of~$\varepsilon\in\{1,2\}$ in~$\mathcal{X}^{\delta}$ ($\delta\in\{A,B\}$) after leaving the initial state:
\begin{align}\label{e:defsigde}
\sigma^{\delta}_{\varepsilon}:=&\begin{cases}\inf\big\{t>0\,\big|\,\exists u\in]0,t[\ \mathcal{X}^{\delta}_u\neq \varepsilon,\ \mathcal{X}^{\delta}_t=\varepsilon\big\},&\textrm{ if }\mathcal{X}^{\delta}_0=\varepsilon,\\\inf\big\{t>0\,\big|\,\mathcal{X}^{\delta}_t=\varepsilon\big\},&\textrm{ otherwise}.\end{cases}
\end{align}
It follows that
\begin{align}\label{e:totalA}
p^A_{1A1}+p^A_{1A2}=1.
\end{align}
$p^A_{2A1}$ and $p^A_{2A2}$ are defined analogously, but conditioned on starting in~$2$. For $\mathcal{X}^B$, the corresponding notations are also introduced, with superscript $B$ in place of~$A$.


For the analogous symbols for the glued chain~$\mathcal{X}$, first define $\sigma_{\varepsilon}$ for $\varepsilon\in\{1,2\}$, through adapting Definition~\eqref{e:defsigde} of $\sigma^{\delta}_{\varepsilon}$,
by the return time to or the hitting time of state~$\varepsilon$. Secondly, let $\delta \in \{A, B\}$ denote the event that the first transition of the glued chain is a transition within chain~$\mathcal{X}^{\delta}$. Then we let
\begin{align*}
p_{1A1}&:=\P(\sigma_1 < \sigma_2, A \,|\, \mathcal{X}_0 = 1),\\
p_{1A2}&:=\P(\sigma_2 < \sigma_1, A \,|\, \mathcal{X}_0 = 1),
\end{align*}
that is, starting from state~$1$, $p_{1A1}$ is the probability that $\mathcal{X}$ leaves $1$ for $\mathcal{S}^A$ and returns to~$1$ before it enters~$2$, while $p_{1A2}$ is the probability that $\mathcal{X}$ leaves $1$ with a transition from $Q^A$ and enters $2$ before it returns to~$1$.

It is straightforward to define the respective quantities with $B$ in place of~$A$. All notations so far introduced for~$\mathcal{X}$ can be recast with $1$ and~$2$ interchanged.


In addition to probabilities, we also need the intensities of leaving on different excursions, so we let
\begin{align}\label{e:defqdede}
q^{\delta}_{\varepsilon \delta \varepsilon'}&:=q_{\varepsilon \delta} \,p^{\delta}_{\varepsilon \delta \varepsilon'}\qquad(\delta\in\{A,B\},\ \varepsilon,\varepsilon'\in\{1,2\}).
\end{align}
We define random variables $\chi^{\delta}_{\varepsilon'}(k)$ for a state $k\in\mathcal{S}^\delta=\mathcal{V}^{\delta}\setminus\{1,2\}$ by
\[\chi^{\delta}_{\varepsilon'}(k):=\int_0^{\sigma^{\delta}_{\varepsilon'}}\mathbf{1} (\mathcal{X}^{\delta}_t=k)\id t.\]
That is, $\chi^{\delta}_{\varepsilon'}(k)$ is the time spent by~$\mathcal{X}^{\delta}$ in a state $k\in\mathcal{S}^\delta$ on a transition path from the initial state until $\mathcal{X}^{\delta}$ reaches state~$\varepsilon'\in\{1,2\}$. We need a notion of conditional expectation of $\chi^{\delta}_{\varepsilon'}(k)$ with initial state~$\varepsilon\in\{1,2\}$, which we define the following way:
\begin{multline*}
\E_{\varepsilon}\Big[\chi^{\delta}_{\varepsilon'}(k)\, \Big|\, \sigma^{\delta}_{\varepsilon'}<\sigma^{\delta}_{3-\varepsilon'}\Big]:=\\
\begin{cases}
\E_{\varepsilon}[\chi^{\delta}_{\varepsilon'}(k)\mathbf{1} (\sigma^{\delta}_{\varepsilon'}<\sigma^{\delta}_{3-\varepsilon'})]\Big/\P_{\varepsilon}(\sigma^{\delta}_{\varepsilon'}<\sigma^{\delta}_{3-\varepsilon'}),&\textrm{if}\ \P_{\varepsilon}(\sigma^{\delta}_{\varepsilon'}<\sigma^{\delta}_{3-\varepsilon'})\neq 0,\\
0,&\textrm{otherwise}.
\end{cases}
\end{multline*}
($\P_{\varepsilon}(\cdot)$, $\E_{\varepsilon}[\cdot]$ are the usual probability, respectively, expectation when the process is started from~$\varepsilon$. The expectations in this paper are all finite because the state spaces are finite and irreducible.)
For compactness, we use the shorthand
\[\E_{\varepsilon\varepsilon'}[\chi^{\delta}_{\varepsilon'}(k)]:=\E_{\varepsilon}\Big[\chi^{\delta}_{\varepsilon'}(k)\, \Big|\, \sigma^{\delta}_{\varepsilon'}<\sigma^{\delta}_{3-\varepsilon'}\Big].\]

\noindent
We introduce the vectors $v,w\in[0,\infty[^{r-2}$, $x,y\in[0,\infty[^{s-2}$ by
\begin{align}
v_i&:=q^A_{1A1}\E_{11}[\chi^A_1(i)]+q^A_{1A2}\E_{12}[\chi^A_2(i)],\notag\\
w_i&:=q^A_{2A2}\E_{22}[\chi^A_2(i)]+q^A_{2A1}\E_{21}[\chi^A_1(i)],&&\textrm{for }i\in\mathcal{S}^A,\notag\\
x_{j}&:=q^B_{1B1}\E_{11}[\chi^B_1(j)]+q^B_{1B2}\E_{12}[\chi^B_2(j)],\notag\\
y_{j}&:=q^B_{2B2}\E_{22}[\chi^B_2(j)]+q^B_{2B1}\E_{21}[\chi^B_1(j)],&&\textrm{for }j\in\mathcal{S}^B,\label{e:defvwxy}
\end{align}
and note that these quantities are given purely in terms of the original Markov chains $\mathcal{X}^A$ and~$\mathcal{X}^B$.

\begin{thm}\label{th:main}
The stationary distribution $\pi\in\mathbb{R}^{r+s-2}$ of the glued Markov jump process~$\mathcal{X}$ is given by the unique solution of the following system of linear equations:
\begin{align}
(\pi_{-r+3},\pi_{-r+4},\dots,\pi_0)\T&=\pi_1 v+\pi_2 w,\label{e:pioncaligSA}\\
(\pi_3,\pi_4,\dots,\pi_s)\T&=\pi_1 x+\pi_2 y,\label{e:pioncaligSB}\\
(q^A_{1A2}+q^B_{1B2})\pi_1&=(q^A_{2A1}+q^B_{2B1})\pi_2,\label{e:pionS1S2}\\
\sum_{i=-r+3}^{s}\pi_i&=1.\label{e:pisumsto1}
\end{align}
Specifically,\small
\begin{align}
\pi_1&=\frac{q^A_{2A1}+q^B_{2B1}}{(q^A_{2A1}+q^B_{2B1})(\sum v_i+\sum x_i +1)+(q^A_{1A2}+q^B_{1B2})(\sum w_i +\sum y_i +1)},\label{e:pirm1}\\
\pi_2&=\frac{q^A_{1A2}+q^B_{1B2}}{(q^A_{2A1}+q^B_{2B1})(\sum v_i+\sum x_i +1)+(q^A_{1A2}+q^B_{1B2})(\sum w_i +\sum y_i +1)},\label{e:pir}
\end{align}
\normalsize while $(\pi_{-r+3},\pi_{-r+4},\dots,\pi_0)$ and $(\pi_3,\pi_4,\dots,\pi_s)$ follow from Eqs.~\eqref{e:pioncaligSA} and~\eqref{e:pioncaligSB}.
\end{thm}

In plain terms, the ratio of $\pi_1/\pi_2$ is dependent on the intensity of leaving state~$2$ towards $1$ versus the intensity of transitioning in the opposite direction. The stationary probability on $k\in\mathcal{S}^A\cup\mathcal{S}^B$ is dependent on the intensity of leaving on an excursion towards its half of the state space times the expected time spent there on any such excursion, weighted by the probability masses $\pi_1$ and~$\pi_2$ of states $1$ and~$2$.


We continue in Section~\ref{s:mainproof} with the proof of this theorem. For applications of the theorem, Section~\ref{s:recursion} discusses how to compute $v,w,x,y$ through $q^{\delta}_{\varepsilon \delta \varepsilon'}$ and $\E_{\varepsilon\varepsilon'}[\chi^{\delta}_{\varepsilon'}(k)]$ ($\delta\in\{A,B\}$, $\varepsilon,\varepsilon'\in\{1,2\}$, $k\in\mathcal{S}^{\delta}$). Section~\ref{s:algo} summarises the required calculations in algorithmic form. The special case which is most similar to gluing at one state (i.e.\ the case when $\pi^A, \pi^B$ and $\pi$ are constant multiples of each other) is examined in Theorem~\ref{th:parallel} and Proposition~\ref{th:equiv} of Section~\ref{s:parallel}. Afterwards, for the complementary case, we give bounds for the stationary distribution in terms of the stationary distributions of the original chains in Theorem~\ref{th:viwifrompi}. The article ends with case studies in Section~\ref{s:examples} and concluding remarks.

\section{Proof of Theorem~\ref{th:main}}\label{s:mainproof}

The argument we give considers the Markov chains $\mathcal{X}$, $\mathcal{X}^A$ and~$\mathcal{X}^B$ as regenerative processes, where the renewal times are defined to be the times when $\mathcal{X}$ (or $\mathcal{X}^A$ or~$\mathcal{X}^B$, respectively) enters state $1$ such that it visited $2$ more recently than~$1$.

The calculation is based on the ergodic theorem for regenerative processes (\citet{Smith1955} or \citet{Serfozo2009}, p123): the weight in the stationary distribution assigned to any one state is proportional to the time the process spends in that state in an average segment of the regenerative process. Between any two renewal events,
\begin{enumerate}
\item $\mathcal{X}$ makes nonnegative numbers of $t_{1A1}$- and $t_{1B1}$-excursions (denoted by $\xi_{1A1}$ and $\xi_{1B1}$, respectively),
\item $\mathcal{X}$ transitions from state~$1$ to state~$2$ via a $t_{1A2}$- or a $t_{1B2}$-excursion,
\item $\mathcal{X}$ makes nonnegative numbers of $t_{2A2}$- and $t_{2B2}$-excursions (denoted by $\xi_{2A2}$ and $\xi_{2B2}$, respectively),
\item $\mathcal{X}$ transitions from state~$2$ to state~$1$ via a $t_{2A1}$- or a $t_{2B1}$-excursion.
\end{enumerate}
Due to the finiteness of state spaces and irreducibility, the nonnegative numbers $\xi_{\varepsilon\delta\varepsilon}$ are finite almost surely. They are drawn from, first, a geometric distribution (we mean the variant of the geometric distribution which can take the value~0: with parameter~$p\in]0,1]$, $\P(Y=k)=(1-p)^k p$ for any nonnegative integer~$k$), which is then partitioned into two (whether the excursions are in $\mathcal{S}^A$ or in~$\mathcal{S}^B$) with a binomial variable.

This regenerative structure is used to describe the behaviour of~$\mathcal{X}$ in terms of the behaviour of $\mathcal{X}^A$ and~$\mathcal{X}^B$. Since if for $\mathcal{X}^A$ (the case of $\mathcal{X}^B$ is similar) renewals are defined analogously by returns to $1$, then the behaviour between any two renewal events is the following:
\begin{enumerate}
\item $\mathcal{X}^A$ makes a geometrically distributed number $\xi^A_{1A1}$ of $t^A_{1A1}$-excursions,
\item $\mathcal{X}^A$ transitions from state~$1$ to state~$2$ via a $t^A_{1A2}$-excursion,
\item $\mathcal{X}^A$ makes a geometrically distributed number $\xi^A_{2A2}$ of $t^A_{2A2}$-excursions,
\item $\mathcal{X}^A$ transitions from state~$2$ to state~$1$ via a $t^A_{2A1}$-excursion.
\end{enumerate}

\subsection{Calculation for the chains $\mathcal X^A$ and $\mathcal X^B$}
Let the random variable $\tau$ be the length of one segment of~$\mathcal{X}$ as a regenerative process (the inter-renewal time), and define $\tau^\delta$ analogously for $\mathcal{X}^\delta$ ($\delta\in\{A,B\}$). Further, let $X^{\delta}(i)$ denote the total time spent by~$\mathcal{X}^{\delta}$ in a state~$i\in\mathcal{V}^{\delta}$ in a complete segment. Using the ergodic theorem for regenerative processes for the first equality in each line, and using the second numbered list for the second equality in each line, for $i\in\mathcal{S}^A$,
\begin{align}
\pi^A_i&=\E[\tau^A]^{-1}\E[X^A(i)]=\E[\tau^A]^{-1}\bigg(\E[\xi^A_{1A1}]\E_{11}[\chi^A_1(i)]+\E_{12}[\chi^A_2(i)]\notag\\
&\quad+\E[\xi^A_{2A2}]\E_{22}[\chi^A_2(i)]+\E_{21}[\chi^A_1(i)]\bigg),\notag\\
\pi^A_1&=\E[\tau^A]^{-1}\E[X^A(1)]=\E[\tau^A]^{-1}\bigg(\E[\xi^A_{1A1}]\frac{1}{q_{1A}}+\frac{1}{q_{1A}}+0+0\bigg),\notag\\
\pi^A_2&=\E[\tau^A]^{-1}\E[X^A(2)]=\E[\tau^A]^{-1}\bigg(0+0+\E[\xi^A_{2A2}]\frac{1}{q_{2A}}+\frac{1}{q_{2A}}\bigg).\label{e:piA}
\end{align}

The parameter of the nonnegative geometric random variable $\xi^A_{1A1}$ is
\[p^A_{1A2}=q^A_{1A2}/q_{1A}\in]0,1]\]
(cf. Eq.~\eqref{e:defqdede}). By Eqs.~\eqref{e:totalA} and~\eqref{e:defqdede}, its mean is
\begin{align*}
\E[\xi^A_{1A1}]&=\frac{1-p^A_{1A2}}{p^A_{1A2}}=\frac{p^A_{1A1}}{p^A_{1A2}}=\frac{q^A_{1A1}}{q^A_{1A2}},
\end{align*}
and similarly,
\begin{align*}
\E[\xi^A_{2A2}]&=\frac{p^A_{2A2}}{p^A_{2A1}}=\frac{q^A_{2A2}}{q^A_{2A1}}.
\end{align*}
Note that $\E[\xi^A_{1A1}]=p^A_{1A1}=0$ is possible, e.g. if the state~$1$ is not connected to any other state than~$2$. All the other analogous mean numbers of returning excursions might be zero for certain state space diagrams. It is also true that $q^{\delta}_{\varepsilon \delta \varepsilon}$ might be zero, but $q^{\delta}_{\varepsilon, \delta, 3-\varepsilon}$ will always be positive because of irreducibility.

Let us substitute the expressions for $\E[\xi^A_{1A1}]$ and $\E[\xi^A_{2A2}]$ into the formulae for $\pi^A_i$, Eqs.~\eqref{e:piA}, and for $\pi^A_1$ and~$\pi^A_2$ use
\begin{align*}
q^A_{1A1}+q^A_{1A2}&=q_{1A},\\
q^A_{2A1}+q^A_{2A2}&=q_{2A}.
\end{align*}
Then, for $i\in\mathcal{S}^A$,
\begin{align*}
\pi^A_i&=\E[\tau^A]^{-1}\bigg(\frac{q^A_{1A1}}{q^A_{1A2}}\E_{11}[\chi^A_1(i)]+\E_{12}[\chi^A_2(i)]\\
&\quad+\frac{q^A_{2A2}}{q^A_{2A1}}\E_{22}[\chi^A_2(i)]+\E_{21}[\chi^A_1(i)]\bigg),\\
\pi^A_1&=\E[\tau^A]^{-1}\bigg(\frac{q^A_{1A1}}{q^A_{1A2}}\,\frac{1}{q_{1A}}+\frac{1}{q_{1A}}\bigg)=\E[\tau^A]^{-1}\frac{1}{q^A_{1A2}},\\
\pi^A_2&=\E[\tau^A]^{-1}\bigg(\frac{q^A_{2A2}}{q^A_{2A1}}\,\frac{1}{q_{2A}}+\frac{1}{q_{2A}}\bigg)=\E[\tau^A]^{-1}\frac{1}{q^A_{2A1}}.
\end{align*}
It follows that
\begin{align}\notag
\pi^A_i&=\pi^A_1q^A_{1A1}\E_{11}[\chi^A_1(i)]+\pi^A_1 q^A_{1A2}\E_{12}[\chi^A_2(i)]\\[6pt]
&\quad+\pi^A_2 q^A_{2A2}\E_{22}[\chi^A_2(i)]+\pi^A_2 q^A_{2A1}\E_{21}[\chi^A_1(i)]\label{eq:4321}
\end{align}
$(i\in\mathcal{S}^A)$. It is also true that
\begin{equation}\label{e:piArm1piAr}
\frac{\pi^A_1}{\pi^A_2}=\frac{q^A_{2A1}}{q^A_{1A2}}.
\end{equation}


We assume that all elementary transition rates are known, and so are $\pi^A_i$ for every~$i$. However, $q^A_{1A1}$ and~$q^A_{1A2}$ are known only to the extent that $q^A_{1A1}+q^A_{1A2}=q_{1A}$. Similarly, $q^A_{2A1}+q^A_{2A2}=q_{2A}$ is all that is known about $q^A_{2A1}$ and~$q^A_{2A2}$.
For $i\in\mathcal{S}^A$, we would like to know
\[\E_{11}[\chi^A_1(i)],\ \E_{12}[\chi^A_2(i)],\ \E_{21}[\chi^A_1(i)],\ \E_{22}[\chi^A_2(i)],\]
but these are not trivially accessible either. In Section~\ref{s:recursion}, we revisit these questions. We write down implicit relations for both sets of unknowns, which result in systems of linear equations whose solutions give both sets of unknowns.

\subsection{Calculation for the glued chain}
We start by establishing some simple connections between transition probabilities in the glued chain and those in the original chains. With the notations of Section~\ref{s:mainresult},
\begin{align}
p_{1A1}&=\P(\sigma_1 < \sigma_2, A \,|\, \mathcal{X}_0 = 1) = \P(\sigma_1 < \sigma_2 \,|\, A, \mathcal X_0 = 1) \, \P(A \,|\,\mathcal X_0 = 1)  \notag\\
&=p^A_{1A1}\,\frac{q_{1A}}{q_{1A}+q_{1B}}=\frac{q^A_{1A1}}{q_{1A}}\,\frac{q_{1A}}{q_{1A}+q_{1B}}=\frac{q^A_{1A1}}{q_{1A}+q_{1B}},\label{e:p1A1}
\end{align}
where the penultimate equality results from Eq.~\eqref{e:defqdede}. Similarly,
\begin{align}\label{e:p1A2}
p_{1A2}&=p^A_{1A2}\,\frac{q_{1A}}{q_{1A}+q_{1B}}=\frac{q^A_{1A2}}{q_{1A}}\,\frac{q_{1A}}{q_{1A}+q_{1B}}=\frac{q^A_{1A2}}{q_{1A}+q_{1B}}.
\end{align}

The equations for the glued chain~$\mathcal{X}$ that correspond to Eqs.~\eqref{e:piA} are, for $i\in\mathcal{S}^A$,
\begin{align*}
\pi_i&=\E[\tau]^{-1}\bigg(\E[\xi_{1A1}]\E_{11}[\chi^A_1(i)]+\frac{p_{1A2}}{p_{1A2}+p_{1B2}}\E_{12}[\chi^A_2(i)]\\
&\quad+\E[\xi_{2A2}]\E_{22}[\chi^A_2(i)]+\frac{p_{2A1}}{p_{2A1}+p_{2B1}}\E_{21}[\chi^A_1(i)]\bigg),\\
\pi_1&=\E[\tau]^{-1}\bigg(\E[\xi_{1A1}+\xi_{1B1}]\frac{1}{q_{1A}+q_{1B}}+\frac{1}{q_{1A}+q_{1B}}+0+0\bigg),\\
\pi_2&=\E[\tau]^{-1}\bigg(0+0+\E[\xi_{2A2}+\xi_{2B2}]\frac{1}{q_{2A}+q_{2B}}+\frac{1}{q_{2A}+q_{2B}}\bigg),
\end{align*}
and for $i\in\mathcal{S}^B$,
\begin{align}
\pi_i&=\E[\tau]^{-1}\bigg(\E[\xi_{1B1}]\E_{11}[\chi^B_1(i)]+\frac{p_{1B2}}{p_{1A2}+p_{1B2}}\E_{12}[\chi^B_2(i)]\notag\\
&\quad+\E[\xi_{2B2}]\E_{22}[\chi^B_2(i)]+\frac{p_{2B1}}{p_{2A1}+p_{2B1}}\E_{21}[\chi^B_1(i)]\bigg).\label{e:pi}
\end{align}

The parameter of the nonnegative geometric random variable $\xi_{1A1}+\xi_{1B1}$ is
\begin{align}\label{e:parxi1A1+xi1B1}
p_{1A2}+p_{1B2}=(q^A_{1A2}+q^B_{1B2})(q_{1A}+q_{1B})^{-1},
\end{align}
the second form is ensured by Eq.~\eqref{e:p1A2}. The geometric random variable is subdivided into two by an independent binomial variable. Hence, also employing
\[p_{1A1}+p_{1B1}+p_{1A2}+p_{1B2}=1\]
and Eqs.~\eqref{e:p1A1} and~\eqref{e:parxi1A1+xi1B1}, the sought means are
\begin{align*}
\E[\xi_{1A1}]&=\frac{1-(p_{1A2}+p_{1B2})}{p_{1A2}+p_{1B2}}\,\frac{p_{1A1}}{p_{1A1}+p_{1B1}}\\
&=\frac{p_{1A1}+p_{1B1}}{p_{1A2}+p_{1B2}}\,\frac{p_{1A1}}{p_{1A1}+p_{1B1}}=\frac{p_{1A1}}{p_{1A2}+p_{1B2}}\\
&=\frac{ q^A_{1A1}(q_{1A}+q_{1B})^{-1}}{(q^A_{1A2}+q^B_{1B2})(q_{1A}+q_{1B})^{-1}}=\frac{q^A_{1A1}}{q^A_{1A2}+q^B_{1B2}},
\end{align*}
and similarly,
\begin{align*}
\E[\xi_{2A2}]&=\frac{p_{2A2}}{p_{2A1}+p_{2B1}}=\frac{q^A_{2A2}}{q^A_{2A1}+q^B_{2B1}},\\
\E[\xi_{1B1}]&=\frac{p_{1B1}}{p_{1A2}+p_{1B2}}=\frac{q^B_{1B1}}{q^A_{1A2}+q^B_{1B2}},\\
\E[\xi_{2B2}]&=\frac{p_{2B2}}{p_{2A1}+p_{2B1}}=\frac{q^B_{2B2}}{q^A_{2A1}+q^B_{2B1}}.
\end{align*}

Now we substitute the expressions for $\E[\xi_{\varepsilon \delta \varepsilon}]$ ($\varepsilon\in\{1,2\}$, $\delta\in\{A,B\}$) into the formulae for~$\pi_i$, Eq.~\eqref{e:pi}. We also use the following consequences of Eqs.~\eqref{e:p1A2} and \eqref{e:parxi1A1+xi1B1}:
\begin{align*}
\frac{p_{1A2}}{p_{1A2}+p_{1B2}}&=\frac{q^A_{1A2}}{q^A_{1A2}+q^B_{1B2}},\\
\frac{p_{2A1}}{p_{2A1}+p_{2B1}}&=\frac{q^A_{2A1}}{q^A_{2A1}+q^B_{2B1}},\\
\frac{p_{1B2}}{p_{1A2}+p_{1B2}}&=\frac{q^B_{1B2}}{q^A_{1A2}+q^B_{1B2}},\\
\frac{p_{2B1}}{p_{2A1}+p_{2B1}}&=\frac{q^B_{2B1}}{q^A_{2A1}+q^B_{2B1}}.
\end{align*}

For $i\in\mathcal{S}^A$,
\begin{align*}
\pi_i&=\E[\tau]^{-1}\bigg(\frac{q^A_{1A1}}{q^A_{1A2}+q^B_{1B2}}\E_{11}[\chi^A_1(i)]+\frac{q^A_{1A2}}{q^A_{1A2}+q^B_{1B2}}\E_{12}[\chi^A_2(i)]\\
&\quad+\frac{q^A_{2A2}}{q^A_{2A1}+q^B_{2B1}}\E_{22}[\chi^A_2(i)]+\frac{q^A_{2A1}}{q^A_{2A1}+q^B_{2B1}}\E_{21}[\chi^A_1(i)]\bigg),\\
\pi_1&=\E[\tau]^{-1}\bigg(\frac{q^A_{1A1}+q^B_{1B1}}{q^A_{1A2}+q^B_{1B2}}\,\frac{1}{q_{1A}+q_{1B}}+\frac{1}{q_{1A}+q_{1B}}\bigg)=\E[\tau]^{-1}\frac{1}{q^A_{1A2}+q^B_{1B2}},\\
\pi_2&=\E[\tau]^{-1}\bigg(\frac{q^A_{2A2}+q^B_{2B2}}{q^A_{2A1}+q^B_{2B1}}\,\frac{1}{q_{2A}+q_{2B}}+\frac{1}{q_{2A}+q_{2B}}\bigg)=\E[\tau]^{-1}\frac{1}{q^A_{2A1}+q^B_{2B1}},
\end{align*}
and for $i\in\mathcal{S}^B$,
\begin{align*}
\pi_i&=\E[\tau]^{-1}\bigg(\frac{q^B_{1B1}}{q^A_{1A2}+q^B_{1B2}}\E_{11}[\chi^B_1(i)]+\frac{q^B_{1B2}}{q^A_{1A2}+q^B_{1B2}}\E_{12}[\chi^B_2(i)]\\
&\quad+\frac{q^B_{2B2}}{q^A_{2A1}+q^B_{2B1}}\E_{22}[\chi^B_2(i)]+\frac{q^B_{2B1}}{q^A_{2A1}+q^B_{2B1}}\E_{21}[\chi^B_1(i)]\bigg).
\end{align*}
Consequently, for $i\in\mathcal{S}^A$,
\begin{align}\notag
\pi_i&=\pi_1 q^A_{1A1}\E_{11}[\chi^A_1(i)]+\pi_1 q^A_{1A2}\E_{12}[\chi^A_2(i)]\\[6pt]
&\quad+\pi_2 q^A_{2A2}\E_{22}[\chi^A_2(i)]+\pi_2 q^A_{2A1}\E_{21}[\chi^A_1(i)],\label{eq:2343}
\end{align}
and for $i\in\mathcal{S}^B$,
\begin{align}\notag
\pi_i&=\pi_1 q^B_{1B1}\E_{11}[\chi^B_1(i)]+\pi_1 q^B_{1B2}\E_{12}[\chi^B_2(i)]\\[6pt]
&\quad+\pi_2 q^B_{2B2}\E_{22}[\chi^B_2(i)]+\pi_2 q^B_{2B1}\E_{21}[\chi^B_1(i)].\label{eq:2344}
\end{align}
It also follows that
\begin{equation}\label{e:pirm1pir}
\frac{\pi_1}{\pi_2}=\frac{q^A_{2A1}+q^B_{2B1}}{q^A_{1A2}+q^B_{1B2}}.
\end{equation}

\subsection{Combining the preceding}

With the vectors $v,w\in[0,\infty[^{r-2}$, $x,y\in[0,\infty[^{s-2}$ introduced in Section~\ref{s:mainresult}, using Eq.~\eqref{eq:4321} for the first two equalities and Eqs. \eqref{eq:2343} and~\eqref{eq:2344} for the third and fourth,
\begin{align}
(\pi^A_{-r+3},\pi^A_{-r+4},\dots,\pi^A_0)\T&=\pi^A_1 v+\pi^A_2 w&&\textrm{for }\mathcal{X}^A,\label{e:piAoncaligVA}\\[6pt]
(\pi^B_3,\pi^B_4,\dots,\pi^B_s)\T&=\pi^B_1 x+\pi^B_2 y&&\textrm{for }\mathcal{X}^B,\notag\\[6pt]
(\pi_{-r+3},\pi_{-r+4},\dots,\pi_0)\T&=\pi_1 v+\pi_2 w&&\textrm{and}\label{e:pi1pirm2}\\[6pt]
(\pi_3,\pi_4,\dots,\pi_s)\T&=\pi_1 x+\pi_2 y&&\textrm{for }\mathcal{X}.\label{e:pirp1rpsm2}
\end{align}
This already shows Eqs.~\eqref{e:pioncaligSA} and~\eqref{e:pioncaligSB} of the theorem. Eq.~\eqref{e:pionS1S2} is a consequence of Eq.~\eqref{e:pirm1pir}, while Eq.~\eqref{e:pisumsto1} must anyway be true.

Let us now show the last two statements of the theorem. The following derivation arrives at an explicit solution, thereby showing the uniqueness of~$\pi$. By summing all entries of Eqs.~\eqref{e:pi1pirm2} and~\eqref{e:pirp1rpsm2}, we get
\begin{align*}
\pi_1 \bigg(\sum_i v_i+\sum_i x_i\bigg)+\pi_2 \bigg(\sum_i w_i+\sum_i y_i\bigg)&=\sum_{i\neq 1,2}\pi_i=1-\pi_1-\pi_2,
\end{align*}
which is equivalent to
\begin{align*}
\pi_1 \bigg(\sum_i v_i+\sum_i x_i +1\bigg)+\pi_2 \bigg(\sum_i w_i+\sum_i y_i +1\bigg)&=1.
\end{align*}
The substitution of $\pi_2$ from Eq.~\eqref{e:pirm1pir} yields
\begin{align*}
\pi_1 \left(\sum_i v_i+\sum_i x_i +1+\frac{q^A_{1A2}+q^B_{1B2}}{q^A_{2A1}+q^B_{2B1}}\bigg(\sum_i w_i+\sum_i y_i +1\bigg)\right)&=1,
\end{align*}
which gives $\pi_1$ of Theorem~\ref{th:main}. Eq.~\eqref{e:pirm1pir} gives~$\pi_2$. This concludes the proof.

\section{Computing $q^{\delta}_{\varepsilon \delta \varepsilon'}$ and $\E_{\varepsilon\varepsilon'}[\chi^{\delta}_{\varepsilon'}(k)]$}\label{s:recursion}

The formulae in Theorem~\ref{th:main} use unknown parameters $q^{\delta}_{\varepsilon \delta \varepsilon'}$ and $\E_{\varepsilon\varepsilon'}[\chi^{\delta}_{\varepsilon'}(k)]$ ($\delta\in\{A,B\}$, $\varepsilon,\varepsilon'\in\{1,2\}$, $k\in\mathcal{S}^{\delta}$) about the chains $\mathcal{X}^A$ and~$\mathcal{X}^B$ to express~$\pi$. However, these unknowns can be computed by linear recursions. We demonstrate the calculation for $\delta=B$ and not for~$A$, as in~$\mathcal{X}^A$, the indexing is unconventional.

As it turns out, the matrix
\[Q^B_0:=\left[\begin{array}{ccccc}Q^B_{11}&0&Q^B_{13}&\dots&Q^B_{1s}\\0&Q^B_{22}&Q^B_{23}&\dots&Q^B_{2s}\\0&0&Q^B_{33}&\dots&Q^B_{3s}\\\vdots&\vdots&\vdots&\ddots&\vdots\\0&0&Q^B_{s3}&\dots&Q^B_{ss}\end{array}\right]\]
plays a central role. We will show in Section~\ref{s:prooffullrank}:
\begin{prop}\label{th:fullrank}
$Q^B_0$ has full rank.
\end{prop}

\subsection{The probabilities of leaving on different excursions and $q^{\delta}_{\varepsilon \delta \varepsilon'}$}

In order to establish the values $q^{\delta}_{\varepsilon \delta \varepsilon'}$, we calculate the corresponding $p^{\delta}_{\varepsilon \delta \varepsilon'}$, and the conversion follows from Eq.~\eqref{e:defqdede}.
We extend the usage of $p^B_{\varepsilon B \varepsilon'}$ to any starting point $i\in\{3,4,\dots,s\}$: for the hitting time $\sigma^B_{\varepsilon'}$ of the state $\varepsilon'\in\{1,2\}$,
\[p^B_{iB\varepsilon'}:=\P(\sigma^B_{\varepsilon'} < \sigma^B_{3-\varepsilon'}\ |\ \mathcal{X}^B_0=i).\]

\begin{prop}\label{P:Qede}
$p^B_{\varepsilon B \varepsilon'}$ ($\varepsilon,\varepsilon'\in\{1,2\}$) can be computed as the unique solutions of the following linear equations:
\begin{align}
Q^B_0 \left(\begin{array}{c}p^B_{1 B 1}\\p^B_{2 B 1}\\p^B_{3 B 1}\\\vdots\\p^B_{s B 1}\end{array}\right)&=\left(\begin{array}{c}0\\-Q^B_{21}\\-Q^B_{31}\\\vdots\\-Q^B_{s1}\end{array}\right),&
Q^B_0 \left(\begin{array}{c}p^B_{1 B 2}\\p^B_{2 B 2}\\p^B_{3 B 2}\\\vdots\\p^B_{s B 2}\end{array}\right)&=\left(\begin{array}{c}-Q^B_{12}\\0\\-Q^B_{32}\\\vdots\\-Q^B_{s2}\end{array}\right).\label{e:pB1}
\end{align}
In practice, it suffices to solve only one of the two linear equations and use an analogue of Eq.~\eqref{e:totalA}. 
\end{prop}

\begin{proof}
We introduce notation for probabilities of direct transitions:
\[p^B_{ij}:=\frac{Q^B_{ij}}{-Q^B_{ii}}\qquad(i,j\in\{1,2,\dots,s\},\ i\neq j)\]
(cf. Eq.~\eqref{e:defqdede}). Then the following relations hold:
\begin{align}
p^B_{\varepsilon' B\varepsilon'}&=\sum_{j\notin\{\varepsilon',3-\varepsilon'\}}p^B_{\varepsilon' j}p^B_{j B\varepsilon'},\label{e:pBrecursion1}\\
p^B_{3-\varepsilon', B,\varepsilon'}&=\sum_{j\notin\{\varepsilon',3-\varepsilon'\}}p^B_{3-\varepsilon', j}p^B_{j B\varepsilon'}+p^B_{3-\varepsilon',\varepsilon'},\label{e:pBrecursion2}\\
p^B_{i B\varepsilon'}&=\sum_{j\notin\{\varepsilon',3-\varepsilon',i\}}p^B_{ij}p^B_{j B\varepsilon'}+p^B_{i\varepsilon'}\qquad (i\in\{3,4,\dots,s\}).\label{e:pBrecursion3}
\end{align}
For instance, Eq.~\eqref{e:pBrecursion2} is equivalent to
\[p^B_{3-\varepsilon', B,\varepsilon'}=\sum_{j\notin\{\varepsilon',3-\varepsilon'\}}\frac{Q^B_{3-\varepsilon', j}}{-Q^B_{3-\varepsilon',3-\varepsilon'}}\,p^B_{j B\varepsilon'}+\frac{Q^B_{3-\varepsilon',\varepsilon'}}{-Q^B_{3-\varepsilon',3-\varepsilon'}}.\]
After multiplying both sides by $-Q^B_{3-\varepsilon',3-\varepsilon'}$ and moving both individual terms to the other side,
\[-Q^B_{3-\varepsilon',\varepsilon'}=\sum_{j\neq\varepsilon'}Q^B_{3-\varepsilon', j} p^B_{j B\varepsilon'}.\]
This gives, for $\varepsilon'\in\{1,2\}$,
\begin{align*}
-Q^B_{21}&=[0\ Q^B_{22}\ Q^B_{23}\ \dots\ Q^B_{2s}]\,(p^B_{1 B 1},p^B_{2 B 1},p^B_{3 B 1},\dots,p^B_{s B 1})\T,\\[6pt]
-Q^B_{12}&=[Q^B_{11}\ 0\ Q^B_{13}\ \dots\ Q^B_{1s}]\,(p^B_{1 B 2},p^B_{2 B 2},p^B_{3 B 2},\dots,p^B_{s B 2})\T.
\end{align*}
The same steps of manipulation on Eq.~\eqref{e:pBrecursion1} yield
\begin{align*}
0&=[Q^B_{11}\ 0\ Q^B_{13}\ \dots\ Q^B_{1s}]\,(p^B_{1 B 1},p^B_{2 B 1},p^B_{3 B 1},\dots,p^B_{s B 1})\T,\\[6pt]
0&=[0\ Q^B_{22}\ Q^B_{23}\ \dots\ Q^B_{2s}]\,(p^B_{1 B 2},p^B_{2 B 2},p^B_{3 B 2},\dots,p^B_{s B 2})\T,
\end{align*}
and when carried out on Eq.~\eqref{e:pBrecursion3}, then
\begin{align*}
-Q^B_{i1}&=[0\ 0\ Q^B_{i3}\ \dots\ Q^B_{is}]\,(p^B_{1 B 1},p^B_{2 B 1},p^B_{3 B 1},\dots,p^B_{s B 1})\T,\\[6pt]
-Q^B_{i2}&=[0\ 0\ Q^B_{i3}\ \dots\ Q^B_{is}]\,(p^B_{1 B 2},p^B_{2 B 2},p^B_{3 B 2},\dots,p^B_{s B 2})\T,
\end{align*}
for $i\in\{3,4,\dots,s\}$. The last three displays are exactly as the two claimed linear equations in the assertion. The existence of a solution follows from the underlying probabilistic interpretation (i.e. the probabilities $(p^B_{iB\varepsilon'})_{i\in\{1,2,\dots,s\}}$ for $\varepsilon'\in\{1,2\}$ are particular solutions). Uniqueness also holds since $Q^B_0$ does not have a nontrivial nullvector, as proven in Proposition~\ref{th:fullrank}.
\end{proof}

\subsection{The expected total time spent in individual states on excursions, $\E_{\varepsilon\varepsilon'}[\chi^{\delta}_{\varepsilon'}(k)]$}


Again, the computation of $\E_{\varepsilon\varepsilon'}[\chi^{\delta}_{\varepsilon'}(k)]$ requires the usage of the matrix~$Q_0^B$. First, we need to extend its meaning to a general starting point $i\in\{3,4,\dots,s\}$:
\[\E_{i\varepsilon'}[\chi^B_{\varepsilon'}(k)]=\E_i\Big[\chi^B_{\varepsilon'}(k)\, \Big|\, \sigma^B_{\varepsilon'}<\sigma^B_{3-\varepsilon'}\Big],\]
which is again defined to be zero when the condition has zero probability.

\begin{prop}\label{th:EchiB}
$\E_{\varepsilon\varepsilon'}[\chi^B_{\varepsilon'}(j)]$ ($\varepsilon,\varepsilon'\in\{1,2\}$, $j\in\{3,4,\dots,s\}$) can be computed from the unique solutions $u(j)\in[0,\infty[^s$ of the following linear equations:
\begin{align}
Q^B_0 u(j)&=-p^B_{jB\varepsilon'}\,\mathbf{e}^B_j\qquad(\varepsilon'\in\{1,2\}),\label{e:EchiB}
\end{align}
where $\mathbf{e}^B_j\in\mathbb{R}^s$ is the canonical $j$th unit vector. If $u_i(j)=0$, then $\E_{i\varepsilon'}[\chi^B_{\varepsilon'}(j)]=0$. If $u_i(j)>0$, then
\[\E_{i\varepsilon'}[\chi^B_{\varepsilon'}(j)]=\frac{u_i(j)}{p^B_{iB\varepsilon'}}.\]
\end{prop}


\begin{proof}
Fix $j\in\{3,4,\dots,s\}$. We will see that we can write down linear relationships among \[u_i(j)=\E_{i}[\chi^B_{\varepsilon'}(j) \mathbf{1}(\sigma^B_{\varepsilon'}<\sigma^B_{3-\varepsilon'})]=p^B_{iB\varepsilon'} \E_{i\varepsilon'}[\chi^B_{\varepsilon'}(j)]\]
($i\in\{1,2,\dots,s\}$) most naturally. We note that if $p^B_{iB\varepsilon'}=0$, then $\E_{i\varepsilon'}[\chi^B_{\varepsilon'}(j)]=0$ by the definition of the latter as a conditional expectation that is zero when the condition has zero probability.


We can observe the following relationships:
\begin{align}
\E_{\varepsilon'}[\chi^B_{\varepsilon'}(j) \mathbf{1}(\sigma^B_{\varepsilon'}<\sigma^B_{3-\varepsilon'})]&=\sum_{k\notin\{\varepsilon',3-\varepsilon'\}}p^B_{\varepsilon' k}\,\E_{k}[\chi^B_{\varepsilon'}(j) \mathbf{1}(\sigma^B_{\varepsilon'}<\sigma^B_{3-\varepsilon'})],\label{e:EchiBrecursion1}\\[6pt]
\E_{3-\varepsilon'}[\chi^B_{\varepsilon'}(j) \mathbf{1}(\sigma^B_{\varepsilon'}<\sigma^B_{3-\varepsilon'})]&=\sum_{k\notin\{\varepsilon',3-\varepsilon'\}}p^B_{3-\varepsilon', k}\,\E_{k}[\chi^B_{\varepsilon'}(j) \mathbf{1}(\sigma^B_{\varepsilon'}<\sigma^B_{3-\varepsilon'})],\label{e:EchiBrecursion2}\\[6pt]
\E_{i}[\chi^B_{\varepsilon'}(j) \mathbf{1}(\sigma^B_{\varepsilon'}<\sigma^B_{3-\varepsilon'})]&=\sum_{k\notin\{\varepsilon',3-\varepsilon',i\}}p^B_{ik}\,\E_{k}[\chi^B_{\varepsilon'}(j) \mathbf{1}(\sigma^B_{\varepsilon'}<\sigma^B_{3-\varepsilon'})]\quad(i\notin\{1,2,j\}),\label{e:EchiBrecursion3}\\[6pt]
\E_{j}[\chi^B_{\varepsilon'}(j) \mathbf{1}(\sigma^B_{\varepsilon'}<\sigma^B_{3-\varepsilon'})]&=p^B_{jB\varepsilon'}\,\frac{1}{-Q^B_{jj}}+\sum_{k\notin\{\varepsilon',3-\varepsilon',j\}}p^B_{jk}\,\E_{k}[\chi^B_{\varepsilon'}(j) \mathbf{1}(\sigma^B_{\varepsilon'}<\sigma^B_{3-\varepsilon'})].\label{e:EchiBrecursion4}
\end{align}
A slight complication compared to Eqs.~(\ref{e:pBrecursion1}--\ref{e:pBrecursion3}) is the possibility of there being no transition path from the initial state~$j$ to $\varepsilon'$ that avoids $3-\varepsilon'$. Time spent in~$j$ is only accrued on such transition paths. This is the reason why in Eq.~\eqref{e:EchiBrecursion4}, the mean waiting time $-1/Q^B_{jj}$ is multiplied by $p^B_{jB\varepsilon'}$. In case of nonexistence of the required path or excursion, both sides of this equation are zero.

Manipulations analogous to those applied to Eqs.~(\ref{e:pBrecursion1}--\ref{e:pBrecursion3}) give
\begin{align*}
0&=\sum_{k\neq 3-\varepsilon'}Q^B_{\varepsilon' k}\,\E_{k}[\chi^B_{\varepsilon'}(j) \mathbf{1}(\sigma^B_{\varepsilon'}<\sigma^B_{3-\varepsilon'})],\\
0&=\sum_{k\neq\varepsilon'}Q^B_{3-\varepsilon', k}\,\E_{k}[\chi^B_{\varepsilon'}(j) \mathbf{1}(\sigma^B_{\varepsilon'}<\sigma^B_{3-\varepsilon'})],\\
0&=\sum_{k\notin\{\varepsilon',3-\varepsilon'\}}Q^B_{ik}\,\E_{k}[\chi^B_{\varepsilon'}(j) \mathbf{1}(\sigma^B_{\varepsilon'}<\sigma^B_{3-\varepsilon'})]\quad (i\notin\{1,2,j\}),\\
-p^B_{jB\varepsilon'}&=\sum_{k\notin\{\varepsilon',3-\varepsilon'\}}Q^B_{jk}\,\E_{k}[\chi^B_{\varepsilon'}(j) \mathbf{1}(\sigma^B_{\varepsilon'}<\sigma^B_{3-\varepsilon'})],
\end{align*}
which is the same as the asserted linear system. Once again, the existence of a solution follows from the underlying probabilistic interpretation and uniqueness from Proposition~\ref{th:fullrank}. The case $u_i(j)=0$, when $p^B_{iB\varepsilon'}=0$ or $\E_{i\varepsilon'}[\chi^B_{\varepsilon'}(j)]=0$, has been discussed. If $u_i(j)=p^B_{iB\varepsilon'} \E_{i\varepsilon'}[\chi^B_{\varepsilon'}(j)]>0$, then one can divide it by the nonzero $p^B_{iB\varepsilon'}$ to get $\E_{i\varepsilon'}[\chi^B_{\varepsilon'}(j)]$.
\end{proof}

\subsection{Proof of Proposition~\ref{th:fullrank}}\label{s:prooffullrank}
So far we have been careful not to argue the existence of solutions of the equations in Propositions~\ref{P:Qede} and~\ref{th:EchiB} from $Q^B_0$ having full rank, instead we referred to the probabilistic interpretations. It is because the latter existence result is used to prove full rank itself. Indeed, since $p^B_{jB\varepsilon'}$ and $p^B_{j,B,3-\varepsilon'}$ ($j\in\{3,4,\dots,s\}$) cannot be both zero due to the irreducibility of~$\mathcal{X}^B$, Proposition~\ref{th:EchiB} shows that for every $j\in\{3,4,\dots,s\}$, $\mathbf{e}^B_j$ is in the image of~$Q^B_0$. Moreover, from the definition of~$Q^B_0$, $\mathbf{e}^B_1$ and $\mathbf{e}^B_2$ are its eigenvectors with nonzero eigenvalues. Hence all canonical unit vectors are contained in the image of~$Q^B_0$.


\subsection{Applying Theorem~\ref{th:main}}\label{s:algo}

Given the specification of $\mathcal{X}^A$ and $\mathcal{X}^B$, the stationary distribution $\pi$ of~$\mathcal{X}$ can be computed by the following procedure.
\begin{enumerate}
\item Solve one of the two equations in~\eqref{e:pB1}. From the solution $(p^B_{1B\varepsilon'},p^B_{2B\varepsilon'},\dots,p^B_{sB\varepsilon'})\T$ (either $\varepsilon'=1$ or $\varepsilon'=2$), compute for every $j\in\{1,2,\dots,s\}$, $p^B_{j,B,3-\varepsilon'}=1-p^B_{jB\varepsilon'}$.
\item Repeat the previous step with the corresponding equation for~$\mathcal{X}^A$, and for $\varepsilon'\in\{1,2\}$, get $(p^A_{-r+3,A,\varepsilon'},p^A_{-r+4,A,\varepsilon'},\dots,p^A_{2A\varepsilon'})\T$. Here
\[Q^A_0:=\left[\begin{array}{ccccc}Q^A_{-r+3,-r+3}&\dots&Q^A_{-r+3,0}&0&0\\\vdots&\ddots&\vdots&\vdots&\vdots\\Q^A_{0,-r+3}&\dots&Q^A_{00}&0&0\\Q^A_{1,-r+3}&\dots&Q^A_{10}&Q^A_{11}&0\\Q^A_{2,-r+3}&\dots&Q^A_{20}&0&Q^A_{22}\end{array}\right]\]
must be used and the right-hand sides also change accordingly.
\item Solve Eq.~\eqref{e:EchiB} for both $\varepsilon'\in\{1,2\}$ and for every $j\in\{3,4,\dots,s\}$. Compute $\E_{\varepsilon\varepsilon'}[\chi^B_{\varepsilon'}(j)]$ ($\varepsilon,\varepsilon'\in\{1,2\}$, $j\in\{3,4,\dots,s\}$). 
\item Repeat the previous step with the corresponding equations for~$\mathcal{X}^A$ for every $\varepsilon'\in\{1,2\}$ and every $i\in\{-r+3,-r+4,\dots,0\}$.
\item For every $\delta\in\{A,B\}$ and $\varepsilon,\varepsilon'\in\{1,2\}$, compute $q^{\delta}_{\varepsilon \delta \varepsilon'}$ using Eq.~\eqref{e:defqdede}.
\item Compute $v,w\in[0,\infty[^{r-2}$, $x,y\in[0,\infty[^{s-2}$ by their definitions~\eqref{e:defvwxy}.
\item Compute $\pi_1$ by Eq.~\eqref{e:pirm1} and $\pi_2$ by Eq.~\eqref{e:pir}.
\item Compute the remaining coordinates of $\pi$ by Eqs.~\eqref{e:pioncaligSA} and~\eqref{e:pioncaligSB}.
\end{enumerate}

Ultimately, finding the stationary distribution $\pi$ can be done by solving $2r-3$ linear equations of size $r\times r$ and $2s-3$ of size $s\times s$ (one system in Eq.~\eqref{e:pB1} once for $\mathcal{X}^A$ and~$\mathcal{X}^B$ each and Eq.~\eqref{e:EchiB} $2(r-2)$ times for~$\mathcal{X}^A$ and $2(s-2)$ times for~$\mathcal{X}^B$). Importantly, as the left-hand-side matrices are shared, the $2r-3$ linear systems with~$Q_0^A$ can be computed in parallel (putting the right-hand-side vectors into one matrix), and similarly with the $2s-3$ linear systems with~$Q_0^B$. This contrasts with directly solving the $(r+s-2)\times (r+s-2)$-sized $\pi\T Q=0$ equation. The new method does not offer an improvement in theoretical computational complexity (calculations omitted).


\section{Parallelism}\label{s:parallel}

This section is concerned with the questions when can the stationary distribution $\pi$ be expected to be a constant multiple of~$\pi^A$ on~$\mathcal{V}^A$ or of~$\pi^B$ on~$\mathcal{V}^B$ (i.e. when are the corresponding vectors parallel; Theorem~\ref{th:parallel} and Proposition~\ref{th:equiv}), whether one can occur without the other (no, as proven by Proposition~\ref{th:equiv}), and what can be said when neither of the two holds (at least bounds for~$\pi$ can be given, Theorem~\ref{th:viwifrompi}).

\subsection{The parallel case}
Consider the following three conditions:
\begin{align*}
(\mathcal{A})\qquad\frac{\pi^A_1}{\pi^A_2}&=\frac{\pi^B_1}{\pi^B_2},\\[6pt]
(\mathcal{B}^A)\qquad\frac{\pi_1}{\pi_2}&=\frac{\pi^A_1}{\pi^A_2},\\[6pt]
(\mathcal{B}^B)\qquad\frac{\pi_1}{\pi_2}&=\frac{\pi^B_1}{\pi^B_2}.
\end{align*}
Under Condition~$(\mathcal{A})$, the stationary distribution~$\pi$ of the glued chain can be computed exactly. This case is most similar to gluing at one state.
\begin{thm}\label{th:parallel}
If Condition~$(\mathcal{A})$ holds, then
\[\pi_i=\begin{cases}C\pi^A_i \pi^B_1,&\ \textrm{if}\ i\in\mathcal{S}^A,\\C\pi^A_1 \pi^B_1,&\ \textrm{if}\ i=1,\\C\pi^A_2 \pi^B_1=C\pi^A_1 \pi^B_2,&\ \textrm{if}\ i=2,\\C\pi^A_1 \pi^B_i,&\ \textrm{if}\ i\in\mathcal{S}^B,\end{cases}\]
with
\[C=\left(\pi^A_1+\pi^B_1-\pi^A_1(\pi^B_1+\pi^B_2)\right)^{-1}=\left(\pi^A_1+\pi^B_1-\pi^B_1(\pi^A_1+\pi^A_2)\right)^{-1}.\]
In particular, the explicit solution shows that Condition~$(\mathcal{A})$ implies Conditions $(\mathcal{B}^A)$ and~$(\mathcal{B}^B)$.
\end{thm}

\begin{proof}
Three statements need to be verified: the entries of $\pi$ are nonnegative, they sum to one and $\pi\T Q=0$.

To the first, because of irreducibility all entries of $\pi^A$ and~$\pi^B$ are positive. All is needed is that $C>0$. Since $\pi^B$ sums to one,
\[\pi^A_1+\pi^B_1-\pi^A_1(\pi^B_1+\pi^B_2)>\pi^A_1+\pi^B_1-\pi^A_1>0.\]

As to the second statement,
\begin{align*}
\frac{1}{C}\sum_{i=-r+3}^{s}\pi_i&=\sum_{i=-r+3}^2 \pi^A_i \pi^B_1 +\sum_{i=3}^{s} \pi^A_1 \pi^B_i\\
&=1 \pi^B_1 + \pi^A_1 \sum_{i=3}^{s} \pi^B_i\\
&= \pi^B_1 + \pi^A_1 (1-\pi^B_1-\pi^B_2)\\
&=\pi^B_1 + \pi^A_1 -\pi^A_1 (\pi^B_1 +\pi^B_2)
\end{align*}
is exactly $1/C$, proving that the entries of~$\pi$ sum to one.

The validity of the last statement can be most easily seen by considering how the transition rate matrix $Q$ arises from $Q^A$ and~$Q^B$ almost as a block diagonal matrix if not for a $2\times 2$ overlap where the entries are summed. In the following, we rely on $(\pi^A)\T Q^A=0$ and $(\pi^B)\T Q^B=0$. The result of the multiplication $(\pi\T Q)_i$ for $i\in\{-r+3,-r+4,\dots,0\}$ can be seen to be zero from
\begin{align*}
(C^{-1}\pi\T Q)_i&=\pi^B_1 ((\pi^A)\T Q^A)_i=\pi^B_1 0=0.
\end{align*}
The case of $i\in\{3,4,\dots,s\}$ is entirely similar. For $i=1$, using Condition~$(\mathcal{A})$ in the form $\pi^B_1 \pi^A_2=\pi^A_1 \pi^B_2$ in the second step,
\begin{align*}
(C^{-1}\pi\T Q)_1&=\sum_{k\in\mathcal{S}^A} \pi^B_1 \pi^A_k Q^A_{k1}+ \pi^B_1 \pi^A_1 (Q^A_{11}+Q^B_{11})\\
&\quad+ \pi^B_1 \pi^A_2(Q^A_{21}+Q^B_{21}) +\sum_{k\in\mathcal{S}^B} \pi^A_1 \pi^B_k Q^B_{k1}\\
&=\pi^B_1 \sum_{k\in\mathcal{V}^A} \pi^A_k Q^A_{k1} + \pi^A_1\sum_{k\in\mathcal{V}^B} \pi^B_k Q^B_{k1}\\
&=\pi^B_1 0 + \pi^A_1 0 =0.
\end{align*}
The last remaining case, $i=2$, once again uses the identity $\pi^B_1 \pi^A_2=\pi^A_1 \pi^B_2$ in the second step to show
\begin{align*}
(C^{-1}\pi\T Q)_2&=\sum_{k\in\mathcal{S}^A} \pi^B_1 \pi^A_k Q^A_{k2}+ \pi^B_1 \pi^A_1 (Q^A_{12}+Q^B_{12})\\
&\quad+ \pi^B_1 \pi^A_2 (Q^A_{22}+Q^B_{22}) +\sum_{k\in\mathcal{S}^B} \pi^A_1 \pi^B_k Q^B_{k2}\\
&=\pi^B_1 \sum_{k\in\mathcal{V}^A} \pi^A_k Q^A_{k2} + \pi^A_1\sum_{k\in\mathcal{V}^B} \pi^B_k Q^B_{k2}\\
&=\pi^B_1 0 + \pi^A_1 0 =0.
\end{align*}
\end{proof}

\begin{prop}\label{th:equiv}
Conditions $(\mathcal{A})$, $(\mathcal{B}^A)$ and~$(\mathcal{B}^B)$ are all equivalent.
\end{prop}
The interesting observation is that Condition~$(\mathcal{B}^A)$ alone implies Condition~$(\mathcal{B}^B)$ (and vice versa). It is obvious that Conditions $(\mathcal{B}^A)$ and~$(\mathcal{B}^B)$ together imply Condition~$(\mathcal{A})$. Combined with Theorem~\ref{th:parallel}, we will get the statement.
\begin{proof}
From Eqs.~\eqref{e:piArm1piAr} and~\eqref{e:pirm1pir}, Conditions~$(\mathcal{B}^A)$ and~$(\mathcal{B}^B)$ are equivalent to
\begin{align*}
(\mathcal{B}^A)\qquad\frac{q^A_{2A1}+q^B_{2B1}}{q^A_{1A2}+q^B_{1B2}}&=\frac{q^A_{2A1}}{q^A_{1A2}},\\[6pt]
(\mathcal{B}^B)\qquad\frac{q^A_{2A1}+q^B_{2B1}}{q^A_{1A2}+q^B_{1B2}}&=\frac{q^B_{2B1}}{q^B_{1B2}}.
\end{align*}
For brevity, we introduce the notations
\begin{align*}
a&:=q^A_{2A1},&c&:=q^B_{2B1},\\
b&:=q^A_{1A2},&d&:=q^B_{1B2}.
\end{align*}
Due to the irreducibility of the Markov chains, $a,b,c,d$ are all positive and the following equivalences hold:
\begin{align*}
&(\mathcal{B}^A)&&\iff &\frac{a+c}{b+d}&=\frac{a}{b}&&\iff& bc&=ad\\[6pt]
&&&\iff &\frac{a+c}{b+d}&=\frac{c}{d}&&\iff &&(\mathcal{B}^B).
\end{align*}
\end{proof}

\subsection{The non-parallel case}

If Condition~$(\mathcal{A})$ does not hold, then our problem formulation can be turned around: assuming that $\pi^A, \pi^B$ and~$\pi$ are all known, $v,w,x,y$ can be computed by a second method and an elementwise bound for~$\pi$ on $\mathcal{S}^A$ and~$\mathcal{S}^B$ in terms of $\pi$ on states $1$ and~$2$ will also result. Although so far we have been trying to reduce finding~$\pi$ to the smaller problems of finding $\pi^A$ and~$\pi^B$, one should not forget that $\pi$ is always available by a Gaussian elimination from the specification of~$\mathcal{X}$.

\begin{thm}\label{th:viwifrompi}
If Condition~$(\mathcal{A})$ fails, then for $i\in\mathcal{S}^A$,
\begin{align*}
v_i&=\frac{\pi^A_2\pi_i - \pi_2\pi^A_i}{\pi^A_2\pi_1 - \pi_2\pi^A_1},\\
w_i&=\frac{\pi^A_1\pi_i - \pi_1\pi^A_i}{\pi^A_1\pi_2 - \pi_1\pi^A_2},
\end{align*}
and for $j\in\mathcal{S}^B$,
\begin{align*}
x_j&=\frac{\pi^B_2\pi_j - \pi_2\pi^B_j}{\pi^B_2\pi_1-\pi_2\pi^B_1},\\
y_j&=\frac{\pi^B_1\pi_j - \pi_1\pi^B_j}{\pi^B_1\pi_2-\pi_1\pi^B_2}.
\end{align*}
Additionally,
\begin{align}
\label{e:piipiAi}\min\left\{\frac{\pi_1}{\pi^A_1},\frac{\pi_2}{\pi^A_2}\right\}&\le\frac{\pi_i}{\pi^A_i}\le\max\left\{\frac{\pi_1}{\pi^A_1},\frac{\pi_2}{\pi^A_2}\right\},\\[6pt]
\label{e:pijpiBj}\min\left\{\frac{\pi_1}{\pi^B_1},\frac{\pi_2}{\pi^B_2}\right\}&\le\frac{\pi_j}{\pi^B_j}\le\max\left\{\frac{\pi_1}{\pi^B_1},\frac{\pi_2}{\pi^B_2}\right\}
\end{align}
and
\begin{align}\label{e:pirm1pirineq}
\min\left\{\frac{\pi^A_1}{\pi^A_2},\frac{\pi^B_1}{\pi^B_2}\right\}&<\frac{\pi_1}{\pi_2}<\max\left\{\frac{\pi^A_1}{\pi^A_2},\frac{\pi^B_1}{\pi^B_2}\right\}.
\end{align}
\end{thm}
Of the two inequalities in~\eqref{e:piipiAi}, one must be strict, since Condition~$(\mathcal{B}^A)$ fails by assumption and the terms at the two ends are not equal. They are both strict inequalities if and only if both $v_i,w_i$ are positive. Note that when Condition~$(\mathcal{A})$ does hold, then due to Theorem~\ref{th:parallel}, the two inequalities in~\eqref{e:piipiAi} are both equalities and $\pi_i/\pi^A_i=\pi_1/\pi^A_1=\pi_2/\pi^A_2=C\pi^B_1$. Analogously for \eqref{e:pijpiBj}. Lastly, because of Proposition~\ref{th:equiv}, the relation~\eqref{e:pirm1pirineq} also holds with equalities.

\begin{proof}
Let us subtract from the $i$th entry of both sides of Eq.~\eqref{e:pioncaligSA} $\pi_2/\pi^A_2$ times the $i$th entry of both sides of Eq.~\eqref{e:piAoncaligVA}:
\begin{align*}
\pi_i-\frac{\pi_2}{\pi^A_2}\pi^A_i&=\bigg(\pi_1-\frac{\pi_2}{\pi^A_2}\pi^A_1\bigg)v_i.
\end{align*}
Since Condition~$(\mathcal{A})$ does not hold, nor does Condition~$(\mathcal{B}^A)$, and we can divide over to arrive at
\begin{align*}
v_i&=\bigg(\pi_i-\frac{\pi_2}{\pi^A_2}\pi^A_i\bigg)\bigg(\pi_1-\frac{\pi_2}{\pi^A_2}\pi^A_1\bigg)^{-1}=\frac{\pi^A_2\pi_i-\pi_2\pi^A_i}{\pi^A_2\pi_1-\pi_2 \pi^A_1}.
\end{align*}

If we now subtract from the $i$th entry of both sides of Eq.~\eqref{e:pioncaligSA} $\pi_1/\pi^A_1$ times the $i$th entry of both sides of Eq.~\eqref{e:piAoncaligVA}, then
\begin{align*}
\pi_i-\frac{\pi_1}{\pi^A_1}\pi^A_i&=\bigg(\pi_2-\frac{\pi_1}{\pi^A_1}\pi^A_2\bigg)w_i.
\end{align*}
Condition~$(\mathcal{B}^A)$ does not hold, and we can divide both sides to get
\begin{align*}
w_i&=\bigg(\pi_i-\frac{\pi_1}{\pi^A_1}\pi^A_i\bigg)\bigg(\pi_2-\frac{\pi_1}{\pi^A_1}\pi^A_2\bigg)^{-1}=\frac{\pi^A_1\pi_i-\pi_1\pi^A_i}{\pi^A_1\pi_2-\pi_1\pi^A_2}.
\end{align*}

From the definitions of $v$ and~$w$ in Eqs.~\eqref{e:defvwxy}, it is clear that at least one of $v_i$ and $w_i$ is positive. Consequently,
\begin{align*}
(\pi^A_2\pi_i-\pi_2\pi^A_i)(\pi^A_2\pi_1-\pi_2\pi^A_1)&\ge 0,\\
(\pi^A_1\pi_i-\pi_1\pi^A_i)(\pi^A_1\pi_2-\pi_1\pi^A_2)&\ge 0,
\end{align*}
and at least one inequality is strict. (If $v_i>0$, then the first inequality is strict; if $w_i>0$, then the second one is strict.)
Here the second factors, formerly the denominators, are negatives of one another, and are nonzero because of the failure of Condition~$(\mathcal{B}^A)$, therefore the enumerators too must have opposite signs (or at most one of them is zero),
\[(\pi^A_2\pi_i-\pi_2\pi^A_i)(\pi^A_1\pi_i-\pi_1\pi^A_i)\le 0.\]
After dividing both sides by the positive~$(\pi^A_i)^2\pi^A_1\pi^A_2$,
\begin{align}\label{e:negprod}
\left(\frac{\pi_i}{\pi^A_i}-\frac{\pi_2}{\pi^A_2}\right) \left(\frac{\pi_i}{\pi^A_i}-\frac{\pi_1}{\pi^A_1}\right)&\le 0.
\end{align}
Condition~$(\mathcal{B}^A)$ does not hold, therefore either $\pi_1/\pi^A_1>\pi_2/\pi^A_2$ or $\pi_1/\pi^A_1<\pi_2/\pi^A_2$. Ineq.~\eqref{e:negprod} gives the bounds~\eqref{e:piipiAi} for~$\pi_i/\pi^A_i$ in either case, even when one factor is zero.

If $v_i>0$, then $\pi_i/\pi^A_i\neq\pi_2/\pi^A_2$ in~\eqref{e:piipiAi}, and if $w_i>0$, then $\pi_i/\pi^A_i\neq\pi_1/\pi^A_1$. If both $v_i,w_i$ are positive, then all inequalities in this derivation are strict. Hence both inequalities in~\eqref{e:piipiAi} are strict. The argument also works in the reverse direction, to show $v_i,w_i$ are both positive from strict inequalities in~\eqref{e:piipiAi}. The statements for $x$, $y$ and the bounds~\eqref{e:pijpiBj} follow similarly.

Finally, let us suppose that the bounds~\eqref{e:pirm1pirineq} do not hold and we derive a contradiction. The starting assumption is that Condition~$(\mathcal{A})$ fails, and so do $(\mathcal{B}^A)$ and~$(\mathcal{B}^B)$, and $\pi^A_1/\pi^A_2$, $\pi^B_1/\pi^B_2$ and $\pi_1/\pi_2$ have three distinct values. The bounds~\eqref{e:pirm1pirineq} will be violated if and only if $\pi_1/\pi_2$ is either greater or less than the other two fractions. Here we examine the case when it is greater (the other case is similar).

From $\max\{\pi^A_1/\pi^A_2,\pi^B_1/\pi^B_2\}<\pi_1/\pi_2$,
\begin{align}\label{e:indir}
\frac{\pi_2}{\pi^A_2}<\frac{\pi_1}{\pi^A_1}&\quad\textrm{and}\quad\frac{\pi_2}{\pi^B_2}<\frac{\pi_1}{\pi^B_1}.
\end{align}
From this by~\eqref{e:piipiAi} and \eqref{e:pijpiBj}, for $i\in\mathcal{S}^A$ and for $j\in\mathcal{S}^B$,
\[\frac{\pi_2}{\pi^A_2}\le\frac{\pi_i}{\pi^A_i}\quad\textrm{and}\quad\frac{\pi_2}{\pi^B_2}\le\frac{\pi_j}{\pi^B_j}.\]
One more algebraic rearrangement yields
\begin{align}
\frac{\pi^A_i}{\pi^A_2}\le\frac{\pi_i}{\pi_2}\quad&\textrm{and}\quad\frac{\pi^B_j}{\pi^B_2}\le\frac{\pi_j}{\pi_2}.\label{e:ineqpair}
\end{align}
Note that the first holds with strict inequality for $i=1$ and the second for $j=1$ too, due to the indirect assumption~\eqref{e:indir}. The contradiction will arise because such a $\pi$ cannot be a nullvector of~$Q$. To show this, we multiply both sides of the left inequalities by the nonnegative~$Q^A_{i2}$ and sum for all $i\in \mathcal{V}^A\setminus\{2\}$, we multiply both sides of the right inequalities by the nonnegative~$Q^B_{j2}$ and sum for all $j\in \mathcal{V}^B\setminus\{2\}$:
\begin{align*}
\frac{1}{\pi^A_2}\sum_{i\in \mathcal{V}^A\setminus\{2\}}\pi^A_i Q^A_{i2}&\le\frac{1}{\pi_2}\sum_{i\in \mathcal{V}^A\setminus\{2\}}\pi_i Q^A_{i2},\\[6pt]
\frac{1}{\pi^B_2}\sum_{j\in \mathcal{V}^B\setminus\{2\}}\pi^B_j Q^B_{j2}&\le\frac{1}{\pi_2}\sum_{j\in \mathcal{V}^B\setminus\{2\}}\pi_j Q^B_{j2}.
\end{align*}
We add $Q^A_{22}$ to both sides of the first inequality and $Q^B_{22}$ to both sides of the second to get
\begin{align}
\label{e:A2}0=\frac{1}{\pi^A_2}\sum_{i\in\mathcal{V}^A}\pi^A_i Q^A_{i2}&\le\frac{1}{\pi_2}\sum_{i\in\mathcal{V}^A}\pi_i Q^A_{i2},\\[6pt]
\label{e:B2}0=\frac{1}{\pi^B_2}\sum_{j\in\mathcal{V}^B}\pi^B_j Q^B_{j2}&\le\frac{1}{\pi_2}\sum_{j\in\mathcal{V}^B}\pi_j Q^B_{j2}.
\end{align}
The left-hand sides of \eqref{e:A2} and \eqref{e:B2} are zero since $(\pi^A)\T Q^A=0$, $(\pi^B)\T Q^B=0$. The sum of the two right-hand sides of \eqref{e:A2} and \eqref{e:B2} is $\pi_2^{-1}(\pi\T Q)_2$. $\pi\T Q=0$ must hold, therefore a contradiction will arise and the proof will be complete if we show that the inequality in \eqref{e:A2} is strict.

If $Q^A_{12}>0$, then from the first inequality of~\eqref{e:ineqpair} with $i=1$ (which is strict for this~$i$),
\[\frac{\pi^A_1}{\pi^A_2}Q^A_{12}<\frac{\pi_1}{\pi_2}Q^A_{12},\]
and \eqref{e:A2} holds with strict inequality.

If $Q^A_{12}=0$, then state~$2$ must still be reachable from state~$1$ in~$\mathcal{X}^A$, therefore there must be a directed path from $1$ to~$2$ whose last step is from some $k\in\mathcal{S}^A$ to~$2$. Then $v_k>0$ by the definition of~$v$ and $Q^A_{k2}>0$. From the earlier result of this proof on the explicit form of $v_k$, the first inequality in~\eqref{e:ineqpair} cannot be an equality for $i=k$,
\[\frac{\pi^A_k}{\pi^A_2}<\frac{\pi_k}{\pi_2},\]
and
\[\frac{\pi^A_k}{\pi^A_2}Q^A_{k2}<\frac{\pi_k}{\pi_2}Q^A_{k2}\]
gives \eqref{e:A2} with strict inequality.

\end{proof}

\section{Examples}\label{s:examples}

We present two related case studies for the application of the results of this paper. We implemented the numerical algorithm of Section~\ref{s:algo} in \textsc{GNU~Octave} with code that is compatible with \textsc{Matlab} (The MathWorks, Inc.).

First, consider the Markov processes given by the transition rate matrices
\begin{align*}
Q^A&=\left[\begin{array}{ccc}-6&4&2\\1&-2&1\\2&0&-2\end{array}\right],&Q^B&=\left[\begin{array}{cccc}-3&1&0&2\\3&-3&0&0\\0&4&-4&0\\0&0&1&-1\end{array}\right].
\end{align*}
These processes are visualised in Figure~\ref{f:ex1}. States $3,4\in\mathcal{S}^B$ can be reached only on $t^B_{1B2}$-excursions but not on a $t^B_{2B1}$-excursion. Here $\pi^A=(0.2,0.4,0.4)\T$ and $\pi^B=(2/9,2/9,1/9,4/9)\T$. One can notice that Condition~$(\mathcal{A})$ holds:
\[\frac{\pi^A_1}{\pi^A_2}=\frac{0.4}{0.4}=1=\frac{2/9}{2/9}=\frac{\pi^B_1}{\pi^B_2}.\]
Thus Theorem~\ref{th:parallel} holds and the stationary distribution $\pi=(0.1,0.2,0.2,0.1,0.4)\T$ (computed either from $\pi\T Q=0$ or with the explicit solution in Theorem~\ref{th:parallel} or by the numerical algorithm) is indeed parallel to $\pi^A$ on $\mathcal{V}^A$ and to $\pi^B$ on~$\mathcal{V}^B$. See also Figure~\ref{f:exc}.
\begin{figure}[h]
\begin{center}
\includegraphics[height=3cm]{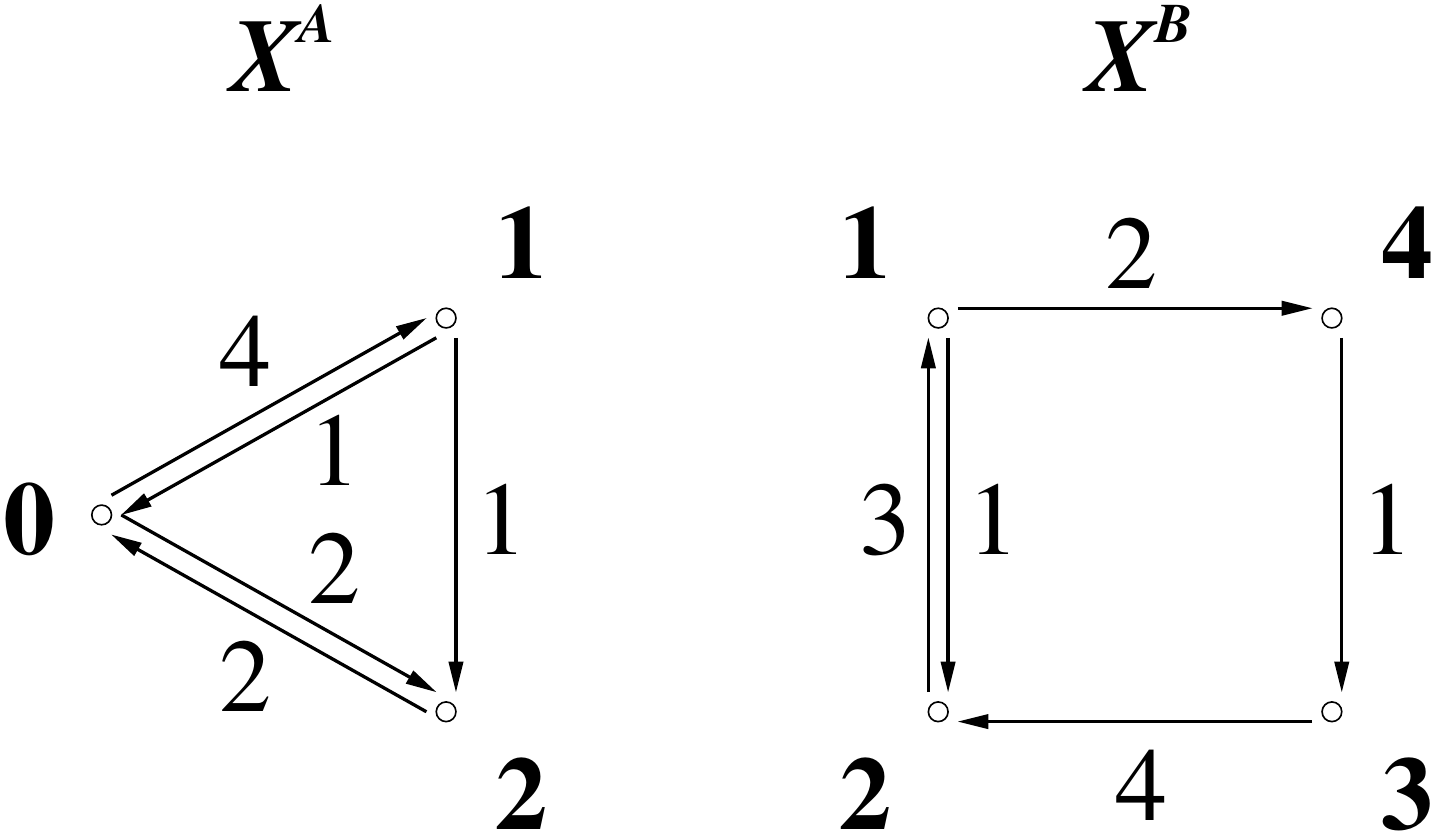}
\caption{State spaces and transitions of $\mathcal{X}^A$ and $\mathcal{X}^B$ of the first example.}
\label{f:ex1}
\end{center}
\end{figure}

The second example is defined by
\begin{align*}
Q^A&=\left[\begin{array}{ccccc}-4&0&0&0&4\\ 1&-1&0&0&0\\ 0&0&-6&4&2\\ 0&2&1&-3&0\\ 0&0&2&0&-2\end{array}\right],&Q^B&=\left[\begin{array}{cc}-2&2\\ 3&-3\end{array}\right],
\end{align*}
and is displayed in Figure~\ref{f:ex2}. In this case $\pi^A=(1/12,1/3,1/8,1/6,7/24)\T$ and $\pi^B=(0.6,0.4)\T$. Condition~$(\mathcal{A})$ does not hold. However, the glued chain has the same state space diagram as in the first example, although the state labels are permuted. Both the direct calculation and the new algorithm give $\pi=(0.1,0.4,0.1,0.2,0.2)\T$. (It also follows from the result for the first example.)

\begin{figure}[h]
\begin{center}
\includegraphics[height=3cm]{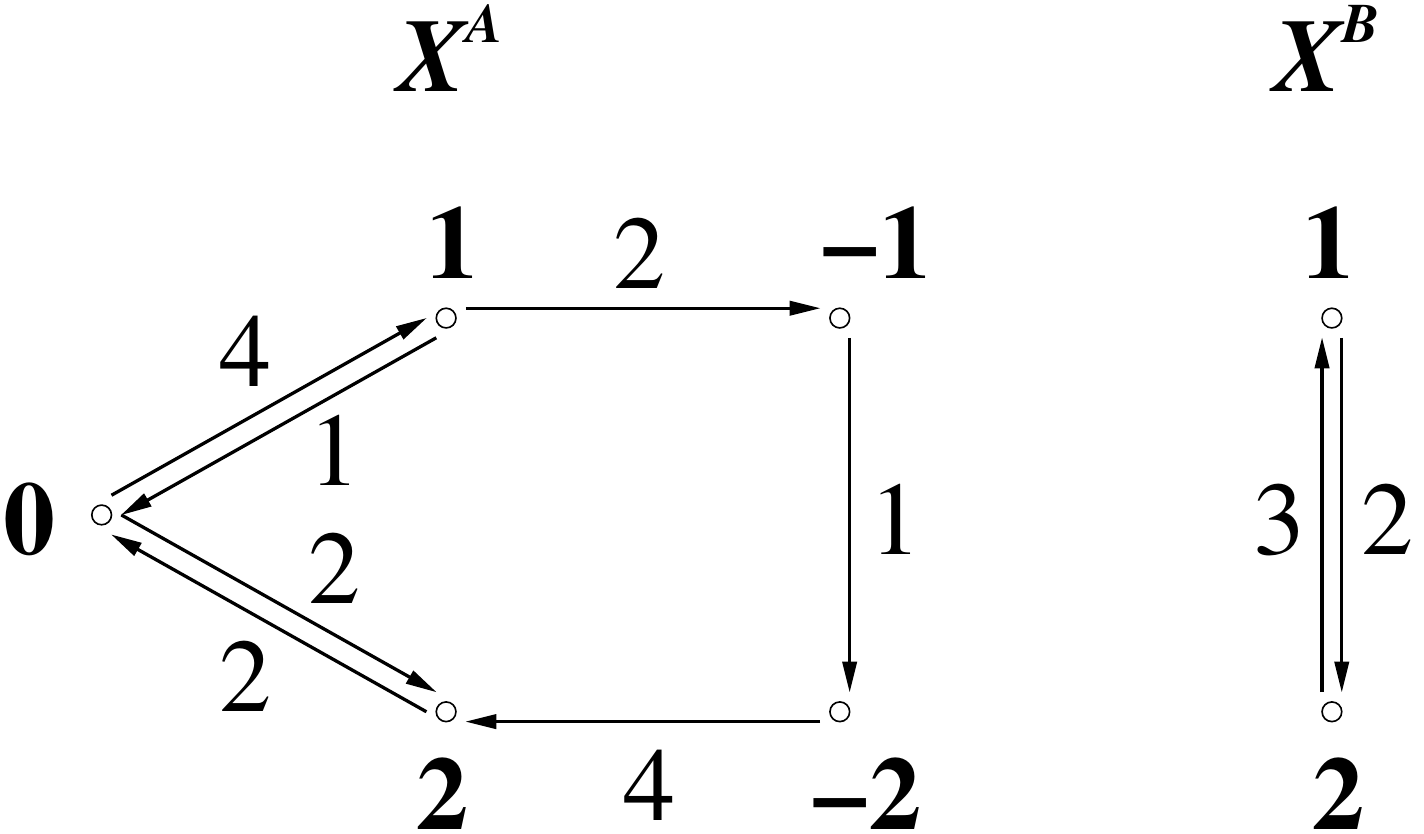}
\caption{State spaces and transitions of $\mathcal{X}^A$ and $\mathcal{X}^B$ of the second example.}
\label{f:ex2}
\end{center}
\end{figure}

Ineq.~\eqref{e:piipiAi} has $\pi_2/\pi^A_2=48/70\approx 0.69$ and $\pi_1/\pi^A_1=12/10$ on its left and right ends, respectively. $(\pi_i/\pi^A_i)_{i\in\{-2,-1,0\}}=(1.2,1.2,0.8)\T$ are sandwiched by those values, as predicted. In addition, the numerical algorithm finds that $v(-2), v(-1), v(0)$ and $w(0)$ are positive, $w(-2)=w(-1)=0$. Hence the pattern of strict inequalities satisfies the observations made in the proof of Theorem~\ref{th:viwifrompi}. Ineq.~\eqref{e:pijpiBj} is absent because $\mathcal{S}^B$ is empty. Lastly, Ineq.~\eqref{e:pirm1pirineq} holds in the form
\[\frac{\pi^A_1}{\pi^A_2}=\frac{4}{7}<\frac{\pi_1}{\pi_2}=1<\frac{\pi^B_1}{\pi^B_2}=\frac{3}{2}.\]

\begin{figure}[h]
\begin{center}
\includegraphics[height=2.96cm]{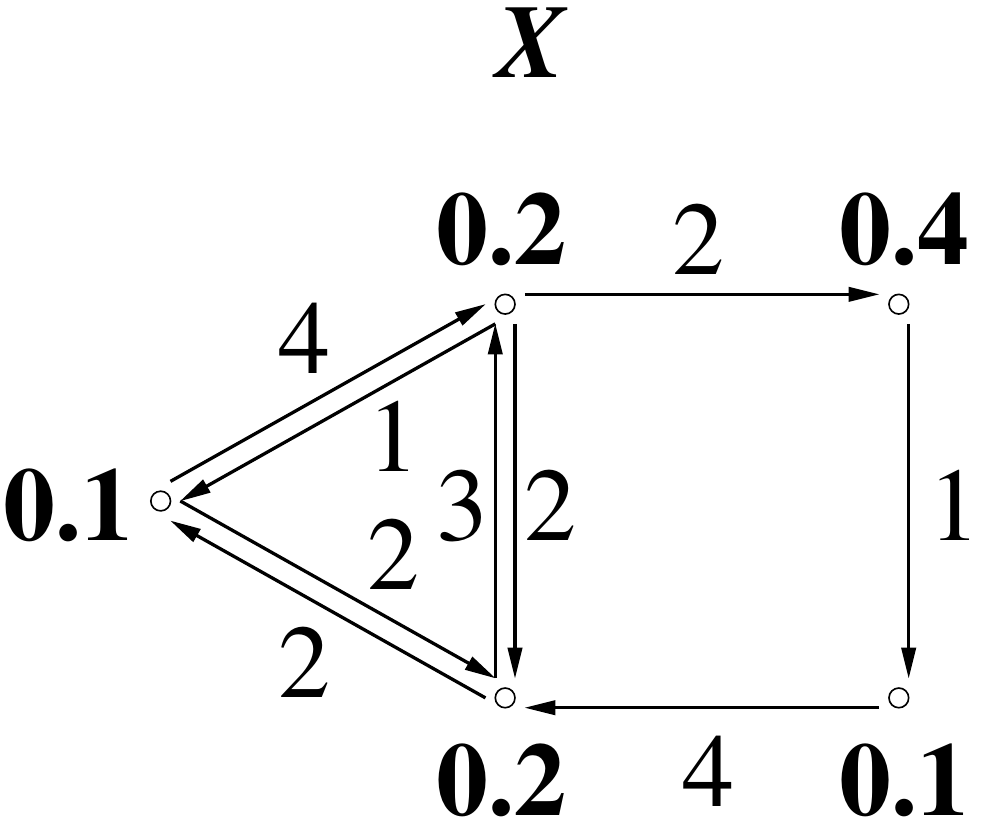} 
\caption{State space and transitions of the glued chain~$\mathcal{X}$, which are identical in the two examples. State labels were omitted as they differ in the two cases and were replaced by the values of the stationary distribution.}
\label{f:exc}
\end{center}
\end{figure}

\section{Conclusion}

This work describes the stationary distribution~$\pi$ in a finite-state, time-homogeneous, continuous-time Markov jump process that was created by gluing together two irreducible Markov processes at two states. When the two original processes share the ratio between the stationary probabilities of their two to-be-glued states (Condition~$(\mathcal{A})$), then the stationary distribution~$\pi$ can be explicitly given (Theorem~\ref{th:parallel}). It is a constant multiple of the original stationary distribution $\pi^A$ on the part of the state space that came from this first process, and a constant multiple of $\pi^B$ on the part of the state space that came from the second process.

When Condition~$(\mathcal{A})$ does not hold, then $\pi$ is known in terms of transition rates in the original chains and additional information about mean times spent in states on excursions from the two glued states and about probabilities of different excursions when leaving the two glued states (Theorem~\ref{th:main}). Ultimately, the knowledge of the original stationary distributions does not suffice to compute~$\pi$.

Also in this case, the ratio between the stationary probabilities of the two glued states in the glued chain is sandwiched between those ratios of the corresponding states in $\mathcal{X}^A$ and of those in~$\mathcal{X}^B$. Other sandwiching bounds are also proven for states in $\mathcal{S}^A\cup\mathcal{S}^B$.

If gluing is applied to grow a large $\mathcal{X}^A$ by an $\mathcal{X}^B$ with a linear state diagram, although the required excursion probabilities etc. might well be computable by hand for~$\mathcal{X}^B$, the same quantities for~$\mathcal{X}^A$ are still only available by solving linear equations, as summarised in Section~\ref{s:algo}. There is no escaping the lengthy calculations for the complicated chain, even when this chain is perturbed by a simple chain.


The regeneration argument exposed in this paper should be applicable when, say, $\mathcal{X}^B$~is not irreducible only because there is no opportunity for $t^B_{2B1}$-excursions (meaning that there are two separate communicating classes in the state space, one containing~$1\in\mathcal{V}^B$ and the other~$2\in\mathcal{V}^B$, and the communicating class of $2$ is an absorbing set), but the gluing introduces such a state-$2$-to-state-$1$ transition via~$\mathcal{V}^A$.

The gluing studied here is a binary operation on the set of Markov chains. A possible line of future research might consider other operations with Markov chains: gluing at multiple states, removing parts of the state space, taking a product of state spaces, or combining product and merging, as was alluded to in Section~\ref{s:intro}. If we stay with gluing at two states, it would be interesting to know how other properties of a Markov chain, such as the mixing time, are affected if another chain is glued to it.

\subsubsection*{Acknowledgments}
B.M. gratefully acknowledges funding by a postdoctoral fellowship of the Alexander von Humboldt Foundation. We are grateful to Andr\'as Gy\"orgy (Massachusetts Institute of Technology), who suggested to write down the linear systems solved in Propositions~\ref{P:Qede} and~\ref{th:EchiB}. We acknowledge Alison Etheridge (University of Oxford) to have noticed that the main points of Theorem~\ref{th:parallel} must hold.

\bibliographystyle{plainnat}

\begin{thebibliography}{10}
\providecommand{\natexlab}[1]{#1}
\providecommand{\url}[1]{\texttt{#1}}
\expandafter\ifx\csname urlstyle\endcsname\relax
  \providecommand{\doi}[1]{doi: #1}\else
  \providecommand{\doi}{doi: \begingroup \urlstyle{rm}\Url}\fi

\bibitem[Adan and Resing(1996)]{Adan_Resing_1996}
Ivo Adan and Jacques Resing.
\newblock Circular {M}arkov chains.
\newblock Technical Report Memorandum COSOR 96-16, Eindhoven University of
  Technology, Eindhoven, The Netherlands, 1996.
\newblock URL \url{http://alexandria.tue.nl/repository/books/461817.pdf}.

\bibitem[Gyorgy and {D}el {V}ecchio(2014)]{Gyorgy_DelVecchio_2014}
Andras Gyorgy and Domitilla {D}el {V}ecchio.
\newblock Modular composition of gene transcription networks.
\newblock \emph{{PLoS} Computational Biology}, 10\penalty0 (3):\penalty0
  e1003486, Mar 2014.
\newblock \doi{10.1371/journal.pcbi.1003486}.

\bibitem[Karlin and Taylor(1975)]{KarlinTaylor1975}
S.~Karlin and H.~M. Taylor.
\newblock \emph{A first course in stochastic processes}.
\newblock Academic Press London, 1975.

\bibitem[Leighton and Rivest(1986)]{Leighton_Rivest_1986}
F.~Leighton and Ronald~L. Rivest.
\newblock Estimating a probability using finite memory.
\newblock \emph{IEEE Transactions on Information Theory}, 32\penalty0
  (6):\penalty0 733--742, Nov 1986.
\newblock ISSN 0018-9448.
\newblock \doi{10.1109/TIT.1986.1057250}.

\bibitem[Liggett(2010)]{Liggett2010}
T.~Liggett.
\newblock \emph{Continuous time {M}arkov processes: {A}n introduction}.
\newblock American Mathematical Society, 2010.

\bibitem[M\'{e}ly\-k\'{u}ti et~al.(2014)M\'{e}ly\-k\'{u}ti, Hespanha, and
  Khammash]{Melykuti_Hespanha_Khammash_2014}
Bence M\'{e}ly\-k\'{u}ti, Jo\~ao~P. Hespanha, and Mustafa Khammash.
\newblock Equilibrium distributions of simple biochemical reaction systems for
  time-scale separation in stochastic reaction networks.
\newblock \emph{Journal of the Royal Society Interface}, 11\penalty0
  (97):\penalty0 20140054, Aug 2014.
\newblock \doi{10.1098/rsif.2014.0054}.

\bibitem[Norris(1998)]{Norris1998}
J.~R. Norris.
\newblock \emph{Markov chains}.
\newblock Cambridge University Press, 1998.

\bibitem[Roginsky(1994)]{Roginsky1994}
Allen~L. Roginsky.
\newblock A central limit theorem for cumulative processes.
\newblock \emph{Advances in Applied Probability}, 26\penalty0 (1):\penalty0
  104--121, Mar 1994.
\newblock URL \url{http://www.jstor.org/stable/1427582}.

\bibitem[Serfozo(2009)]{Serfozo2009}
Richard Serfozo.
\newblock \emph{Basics of applied stochastic processes}.
\newblock Springer, 2009.

\bibitem[Smith(1955)]{Smith1955}
W.~L. Smith.
\newblock Regenerative stochastic processes.
\newblock \emph{Proc. Roy. Soc. Lond. Ser. A.}, 232:\penalty0 6--31, 1955.

\end{thebibliography}

\end{document}